\documentclass[11pt,oneside,leqno]{amsart}

\makeatletter
\@namedef{subjclassname@2020}{\textup{2020} Mathematics Subject Classification}
\makeatother

\usepackage[top=2.5 cm,bottom=2 cm,left=2 cm,right=3 cm]{geometry}
\usepackage[utf8]{inputenc}
\usepackage{color}
\usepackage{amssymb}

\usepackage{amsmath}
\usepackage{amsfonts}
\usepackage{amssymb}

\usepackage[english]{babel}
\usepackage{enumerate} 

\usepackage{todonotes}
\usepackage[backref=page]{hyperref}

\usepackage{esint}
\usepackage{graphicx}
\usepackage{hyperref}
\usepackage{bm}
\usepackage{xcolor}

\newcommand*{\mailto}[1]{\href{mailto:#1}{\nolinkurl{#1}}}

\numberwithin{equation}{section}

\newtheorem{example}{Example}[section]

\newtheorem{theorem}[example]{Theorem}
 \newtheorem{proposition}[example]{Proposition}
\newtheorem{lemma}[example]{Lemma}
 \newtheorem{corollary}[example]{Corollary}
\newtheorem{remark}[example]{Remark}
\newtheorem*{maintheorem*}{Main Theorem}
\allowdisplaybreaks
\numberwithin{equation}{section}

\renewcommand{\i}{\ifmmode\mathit{\mathchar"7010 }\else\char"10 \fi}
\renewcommand{\j}{\ifmmode\mathit{\mathchar"7011 }\else\char"11 \fi}
\newcommand{\R}{\mathbb{R}}
\newcommand{\N}{\mathbb{N}}

\DeclareMathOperator*{\sign}{sign}
\newcommand{\sgn}[1]{\sign\left(#1\right)}

\newcommand{\px}{\partial_x}
\newcommand{\pt}{\partial_t}
\newcommand{\pxi}{\partial_\xi}
\newcommand{\peta}{\partial_\eta}

\newcommand{\ve}{\varepsilon}

{%

\begin{enumerate}}%
{\end{enumerate}}

{%

\begin{enumerate}}%
{\end{enumerate}}

\begin{document}

\title[Well-posedness of the two-component peakon system]
{On the well-posedness of the Cauchy problem for 
the two-component peakon system in $C^k\cap W^{k,1}$}

 \author[Karlsen]{K. H. Karlsen}
\author[Rybalko]{Ya. Rybalko}

\address[Kenneth Hvistendahl Karlsen]{
	Department of Mathematics, University of Oslo, 
	NO-0316 Oslo, Norway}
\email[]{kennethk@math.uio.no}

\address[Yan Rybalko]{
	Mathematical Division, 
	B.Verkin Institute for Low Temperature Physics and Engineering
	of the National Academy of Sciences of Ukraine, 
	47 Nauky Ave., Kharkiv, 61103, Ukraine}
\email[]{rybalkoyan@gmail.com}



\thanks{Yan Rybalko gratefully acknowledges the 
partial support from the Akhiezer Foundation}

\date{\today}

\begin{abstract}
This study focuses on the Cauchy problem associated with the 
two-component peakon system featuring a cubic nonlinearity, 
constrained to the class $(m,n)\in C^{k}(\mathbb{R})
\cap W^{k,1}(\mathbb{R})$ 
with $k\in\mathbb{N}\cup\{0\}$.
This system extends the celebrated Fokas-Olver-Rosenau-Qiao 
equation, and the following nonlocal (two-place) counterpart 
proposed by Lou and Qiao:
$$
\partial_t m(t,x)=
\partial_x[m(t,x)(u(t,x)-\partial_xu(t,x))
(u(-t,-x)+\partial_x(u(-t,-x)))],
$$
where $m(t,x)=\left(1-\partial_{x}^2\right)u(t,x)$. Employing 
an approach based on Lagrangian coordinates, we establish 
the local existence, uniqueness, and Lipschitz continuity of 
the data-to-solution map in the class $C^k\cap W^{k,1}$. 
Moreover, we derive criteria for blow-up of 
the local solution in this class.
\end{abstract}

\maketitle

\tableofcontents

\section{Introduction}
We consider the Cauchy problem for the following 
two-component peakon system with cubic nonlinearity 
introduced by Song, Qu and Qiao in \cite{SQQ11}:
\begin{equation}\label{two-comp}
	\begin{split}
		& \partial_t m=\partial_x[m(u-\partial_xu)(v+\partial_xv)],
		\\
		& \partial_tn=\partial_x[n(u-\partial_xu)(v+\partial_xv)],
		\\
		& m=u-\partial_{x}^2u, \quad n=v-\partial_{x}^2v,
	\end{split}
\end{equation}
for $m=m(t,x),\,u=u(t,x),\,v=v(t,x)$ and $t,x\in \R$. 
We assume that the initial data $u_0(x)=u(0,x)$ and 
$v_0(x)=v(0,x)$ belong to the space 
$C^{k+2}(\mathbb{R})\cap W^{k+2,1}(\mathbb{R})$ with
$k\in\mathbb{N}_0$, where $\mathbb{N}_0=\mathbb{N}\cup\{0\}$.
Taking $u=v$ in \eqref{two-comp}, one obtains 
the Fokas-Olver-Rosenau-Qiao (FORQ) equation,
also referred to as the modified Camassa-Holm equation,
which has the form
\begin{equation}
\begin{split}\label{FORQ}
&\pt m=
\px\left[m\left(u^2-(\px u)^2\right)\right],
\quad m=u-\px^2u.
\end{split}
\end{equation}
Fokas and Fuchssteiner originally introduced this equation 
as an integrable variant of the modified Korteweg-de Vries (mKdV) 
equation, known to possess peakon solutions 
(see \cite[Equation (7)]{F95} and \cite[Equation (26f)]{F81}). 
Subsequently, leveraging its bi-Hamiltonian structure, Olver and Rosenau, 
along with Schiff, derived \eqref{FORQ} as a dual 
counterpart to the mKdV equation, see \cite{OR96} and \cite{Sch96}.
Later, Qiao \cite{Q06} further advanced the development of the FORQ 
equation as an approximation of the two-dimensional Euler equations, 
where $u$ represents the fluid velocity, and $m$ corresponds 
to its potential density. Additionally, \eqref{FORQ} can be 
reduced to the short pulse (SP) equation,
$$
\px\pt u+u+\frac{1}{6}\px^2(u^3)=0,
$$
by the scaling transformation $x'=\frac{x}{\ve}$, 
$t'=\ve t$, $u'=\frac{u}{\ve^2}$ and passing to 
the limit $\ve\to0$ \cite{GLOQ13}.
The SP equation was proposed by Sch\"afer 
and Wayne \cite{SW04} and it is useful for modeling the propagation of
ultra-short light pulses in silica optics.
Lastly, it is worth noting that the FORQ equation can be found 
in the list of equations compiled by Novikov, taking the form 
$(1-\px^2)\pt u=F(u,\px u,\px^2,\dots)$, where $F$ represents 
a quadratic or cubic differential polynomial in $u$ and 
its derivatives with respect to $x$, see \cite[Equation (32)]{N09}.

The FORQ equation, along with its generalizations, has been 
the subject of extensive research, exploring its well-posedness and 
blow-up properties in several works 
\cite{CGLQ16, FGLQ13, GL18, GLOQ13, HM14, Y20, YC24, ZGX22, Zh16}. 
In particular, the geometric formulation of \eqref{FORQ} can 
be found in \cite[Section 2]{GLOQ13}. Additionally, a wide range of 
exact solutions for the FORQ equation, including algebro-geometric, 
peakon, smooth, and loop-shaped solutions, have been derived 
and discussed in various studies \cite{AR19, BKS20, HFQ17, M14, Q06}. 
Moreover, the inverse scattering method and the long-term 
behavior of solutions to the Cauchy problem associated 
with \eqref{FORQ} have been explored in \cite{BKS20, K22}.

Another intriguing reduction of \eqref{two-comp} is the 
nonlocal (two-place) FORQ equation, originally introduced 
by Lou and Qiao \cite[equation (26)]{LQ17}:
\begin{equation}\label{nFORQ-eq}
\begin{split}
&\partial_t m(t,x)=
\partial_x[m(t,x)(u(t,x)-\partial_xu(t,x))
(u(-t,-x)+\partial_x(u(-t,-x)))],\\
&m(t,x)=u(t,x)-\partial_{x}^2u(t,x),
\end{split}
\end{equation}
Equation \eqref{nFORQ-eq} can be linked to \eqref{two-comp} by 
setting $v(t,x)=u(-t,-x)$, thereby allowing us to derive it following 
a methodology akin to that employed by Ablowitz and Musslimani in 
introducing various nonlocal variations of well-known 
integrable equations \cite{AM13,AM17}. Specifically, this approach can 
be applied to obtain the nonlocal counterpart of the 
nonlinear Schr\"odinger (NLS) equation
\begin{equation}\label{NLS}
	\mathrm{i}\partial_tq(t,x)+\partial_{x}^2q(t,x)
	+2\sigma |q(t,x)|^{2}q(t,x)=0,\quad 
	\mathrm{i}^2=-1,\,\,\sigma=\pm1.
\end{equation}
The works \cite{AM13,AM17} considered the following integrable
Ablowitz-Kaup-Newell-Segur (AKNS) system:
\begin{equation*}
\begin{aligned}
& \mathrm{i}\pt q(t,x)+\px^2q(t,x)+2q^2(t,x)r(t,x)=0,\\
-&\mathrm{i}\pt r(t,x)+\px^2r(t,x)+2r^2(t,x)q(t,x)=0,
\end{aligned}
\end{equation*}
which, in the case $r(t,x)=\sigma \overline{q(t,-x)}$, 
reduces to the nonlocal NLS equation
\begin{equation}\label{NNLS}
	\mathrm{i}\partial_{t}q(t,x)+\partial_{x}^2q(t,x)
	+2\sigma q^{2}(t,x)\overline{q(t,-x)}=0.
\end{equation}
The conventional NLS \eqref{NLS} corresponds 
to $r(t,x)=\sigma \overline{q(t,x)}$. It is worth noting 
that \eqref{NNLS} exhibits nonlocal 
behavior exclusively in the spatial variable $x$.

Returning to the nonlocal FORQ equation \eqref{nFORQ-eq}, we 
observe that it incorporates solution values from non-adjacent points, such 
as $x$ and $-x$. This unique feature allows for the description 
of phenomena characterized by intrinsic correlations 
and entanglement between events taking place 
at distinct locations \cite{LH17}. The 
nonlocal FORQ equation \eqref{nFORQ-eq} was initially derived as a 
reduction of the following system, which was introduced 
by Xia, Qiao, and Zhou (see \cite[Equation (7)]{XQZ15}):
\begin{align*}
\nonumber
\label{two-comp-H}
&\partial_tm=\px(mH)+mH-\frac{1}{2}m(u-\partial_xu)
(v+\partial_xv),
\\
&\partial_tn=\px(nH)-nH+\frac{1}{2}[n(u-\partial_xu)
(v+\partial_xv)],
\\
\nonumber
&m=u-\partial_{x}^2u,\,\,n=v-\partial_{x}^2v,
\end{align*}
with $v(t,x)=u(-t,-x)$ and 
$H(t,x)=2(u(t,x)-\px u(t,x))
(u(-t,-x)+\px (u(-t,-x)))$. This system satisfies 
the parity-time-symmetric (PT-symmetric) 
condition, i.e., $H(t,x)=H(-t,-x)$.

The Cauchy problem for the system \eqref{two-comp} can 
be written in the following nonlocal 
form (see \cite[equations (4.1a), (4.1c)]{KR24} and \cite{MM13}):
\begin{subequations}\label{C-nonl-two-comp}
\begin{equation}
\label{nonl-two-comp}
\begin{split}
&\pt u =(1-\px^2)^{-1}\px
[m(u-\partial_xu)
(v+\partial_xv)]\\
&\quad\,\,\,=-\frac{1}{3}(\px u)^2\px v
+\frac{1}{3}(2u(\px u)v+u^2\px v)
+F(u,\px u,v,\px v),\\
&\pt v =(1-\px^2)^{-1}\px
[n(u-\partial_xu)
(v+\partial_xv)]\\
&\quad\,\,\,=-\frac{1}{3}(\px u)(\px v)^2
+\frac{1}{3}(2uv\px v+(\px u)v^2)
+\hat F(u,\px u,v,\px v),
\end{split}
\end{equation}
with initial data
\begin{equation}
u(0,x)=u_0(x),\quad v(0,x)=v_0(x),
\end{equation}
\end{subequations}
where
\begin{equation}
\label{nonl-terms}
\begin{split}
&F(u,w,v,z)=(1-\px^2)^{-1}\left(
\frac{1}{3}w^2z+
\left\{uw\partial_xz-w(\partial_xw)v
+\frac{1}{3}u(u\partial_x^2z-(\partial_x^2w)v)
\right\}
\right)\\
&\qquad\qquad\qquad\,\,+(1-\px^2)^{-1}\partial_x\left(
\frac{2}{3}u^2v+w^2v+B(u,w,v,z)
\right),\\
&\hat F(u,w,v,z)=(1-\px^2)^{-1}\left(
\frac{1}{3}wz^2-
\left\{uz\partial_xz-(\partial_xw)vz
+\frac{1}{3}v(u\partial_x^2z-(\partial_x^2w)v)
\right\}
\right)\\
&\qquad\qquad\qquad\,\,+(1-\px^2)^{-1}\partial_x\left(
\frac{2}{3}uv^2+uz^2+\hat B(u,w,v,z)
\right),
\end{split}
\end{equation}
with
\begin{align*}
	&B(u,w,v,z)=-u(\partial_xw)z+w(\partial_xw)v-uwv+u^2z
	+\frac{1}{3}(w(\partial_xw)z-w^2\partial_xz),\\
	&\hat B(u,w,v,z)=
	wv\partial_xz-uz\partial_xz+uvz-wv^2
	+\frac{1}{3}(wz\partial_xz-(\partial_xw)z^2).
\end{align*}
The Cauchy problem's local well-posedness within a range of 
Besov spaces and its associated blow-up criteria have been 
thoroughly examined in previous studies, see \cite{MM13} 
and \cite{YQZ15} respectively. For additional 
references, see \cite{WY23} and \cite{WZ22}.
Moreover, exact solutions of \eqref{nonl-two-comp}, including 
those involving multipeakons, have been successfully 
derived and investigated in \cite{YQZ15}.
The work  \cite{CHS16} explores the spectral aspects of 
the two-component system, particularly in the context of multipeakons.
Lastly, we mention that the Hamiltonian duality 
between \eqref{two-comp} and other integrable systems has been 
rigorously established in \cite{TL13} and \cite{KR24}.

In our current work, we focus on the study of the Cauchy 
problem \eqref{C-nonl-two-comp} within the class of 
functions $u$ and $v$ belonging to $C([-T,T],C^{k+2}(\R)
\cap W^{k+2,1}(\R))$, where $k\in\N_0$. 
Our approach involves revisiting the method of characteristics, 
as previously developed in \cite{GL18} and \cite{ZGX22}, 
primarily for addressing the FORQ equation.
This method has been successfully applied to various peakon equations, 
enabling one to obtain global solutions that may exhibit finite-time 
singularities, see \cite{HR07,BC07}. Moreover, it has proven valuable in 
the analysis of problems featuring non-zero asymmetric 
asymptotics for $u(t,x)$ as $x$ approaches both positive 
and negative infinity, as demonstrated in \cite{GHR12}.
It is important to note that the Cauchy problem for the nonlocal 
FORQ equation \eqref{nFORQ-eq} can exhibit distinct qualitative properties 
when the asymptotic behavior at positive and negative infinity differs 
significantly, resembling step-like patterns. This distinctive 
behavior arises from the non-translation invariance of \eqref{nFORQ-eq}, in 
contrast to the conventional FORQ equation \eqref{FORQ}.  
We refer to \cite{RS21} for a related discussion for the nonlocal 
NLS equation \eqref{NNLS} with step-like boundary conditions.

By utilizing the explicit representation of the solution $(u,v)$, 
derived from the known initial data and the characteristics, as 
shown in \eqref{uv-repr} below, we establish the existence 
and uniqueness of local solutions within the space $C^{k+2}\cap W^{k+2,1}$.
We note that our chosen class of regularity, specifically for 
the case where $k=0$, exhibits a lower regularity exponent 
than what has been previously explored in works related to the FORQ 
equation and its generalizations, as exemplified in 
\cite{GLOQ13} and \cite{ZGX22}. 
One of the most challenging aspects of our analysis lies in proving 
uniqueness when $k=0$. To achieve this, we must demonstrate 
that every solution adheres to a specific conservation law, as 
outlined in \eqref{cons-uv} below. This entails examining 
equation \eqref{two-comp} in a weak sense, a task that is 
further detailed in Lemma \ref{cl} below.

Next, we establish the Lipschitz continuity property of 
the data-to-solution map for $(m,n)$ within the space $C^k\cap W^{k,1}$, 
where $k\in\N_0$. To the best of our knowledge, this particular 
result has not been previously documented. In related works, such 
as \cite[Section 4.1]{GL18} and \cite[Section 4.1]{Zh16}, it was 
demonstrated that, assuming the existence of a weak solution 
for the FORQ equation in $W^{2,1}$, the solution itself exhibits 
a Lipschitz property within the space $W^{1,1}$. 
For solutions of the FORQ equation residing in $H^s$ with 
$s>\frac{5}{2}$, the data-to-solution map maintains continuity 
but does not possess uniform continuity, as discussed in \cite{HM14}. 
For further insights into the continuity properties of the FORQ equation 
and the two-component system \eqref{two-comp} within the context of 
$H^s$ spaces, we refer to \cite{HM14-1} and \cite{KR24}. 

Finally, we establish new blow-up criteria for the solution 
within the space $C^{k+2}\cap W^{k+2,1}$, 
where $k\in\N_0$. This extends and generalizes previous 
findings presented in \cite{GL18}, which were primarily 
centered on compactly supported classical solutions 
of \eqref{two-comp} with $k\in\N$.

The structure of this article unfolds as follows. In Section \ref{Prelim}, we lay the 
foundation by introducing essential notations, definitions, and 
relevant facts that will serve as the basis for our 
subsequent discussions. Section \ref{Loc} is dedicated to the 
development of the Lagrangian approach for addressing the 
Cauchy problem \eqref{C-nonl-two-comp}. Within this section, we 
leverage Lagrangian coordinates to establish the local existence, 
uniqueness, and Lipschitz continuity of the data-to-solution map. 
A summary of these results is provided in Theorem \ref{lwp}.
Finally, Section \ref{blc} focuses on the task of establishing 
blow-up criteria for the solution in $C^{k+2}\cap W^{k+2,1}$, 
where $k\in\N_0$.

\section{Preliminaries}\label{Prelim}
In this section we introduce some notations 
and facts to be used throughout the paper.
We use the following functional spaces:
\begin{equation}
\nonumber
\begin{split}
&C^{k}(\mathbb{R})=
\left\{f(x):\mathbb{R}\mapsto\mathbb{R}\, \Bigl|\Bigr.\,
f \mbox{ continuous and }
\|f\|_{C^k(\mathbb{R})}\equiv
\sum\limits_{i=0}^{k}
\|\px^i f(x)\|_{C(\mathbb{R})}<\infty\right\},\\
& W^{k,j}(\mathbb{R})=
\left\{f(x)\in L^j(\mathbb{R})\,  \Bigl|\Bigr. \,
\|f\|_{W^{k,j}(\mathbb{R})}\equiv
\sum\limits_{i=0}^{k}
\|\px^i f\|_{L^j(\mathbb{R})}<\infty\right\},
\quad j=1\mbox{ or } j=\infty,
\end{split}
\end{equation}
where $k\in\mathbb{N}_0$.
Also it is convenient for us to use the following 
notations for the Banach spaces
\begin{equation}\label{X-not}
X^0=C(\mathbb{R})\cap L^1(\mathbb{R}),\quad
\mbox{and}\quad
X^k=C^k(\mathbb{R})\cap
W^{k,1}(\mathbb{R}),\,\,
\mbox{for }k\in\mathbb{N}.
\end{equation}
Note that when $f$ belongs to either $C^{k}(\R)$ or
$W^{k,\infty}(\R)$, where $k\in \N$, then $f$ is 
a bounded function. Throughout this text, we adopt the convention 
of writing $L^1$, and similarly for other function spaces, without 
specifying $\R$ when it does not introduce ambiguity to the reader.
Moreover, we use the spaces
$$
C^k\left([-T_1,T_2],X\right),\quad
\mbox{where } k\in\mathbb{N}_0,\,
T_1,T_2>0\mbox{ and }
X \mbox{ is a Banach space},
$$
of $k$-times continuously differentiable functions
$u(t):[-T_1,T_2]\mapsto X$ with the norm
$$
\|u\|_{C^k\left([-T_1,T_2],X\right)}
=\sum\limits_{i=0}^k\sup\limits_{t\in[-T_1,T_2]}
\|\pt^i u(t)\|_X.
$$

Finally, we will use the following elementary inequalities
\begin{subequations}
\begin{align}
\label{exp-ineq}
&|e^a-e^b|\leq|a-b|,&& \mbox{for all } a,b\leq 0,\\
\label{exp-ineq-b}
&|e^a-e^b|\leq e^{\max\{a,b\}}|a-b|,&& \mbox{for all } 
a,b\in\mathbb{R},\\
\label{ab-diff}
&|a_1b_1-a_2b_2|\leq|b_1||a_1-a_2|+|a_2||b_1-b_2|,
&& \mbox{for all } a_j,b_j\in\mathbb{R},\,\,j=1,2.
\end{align}
\end{subequations}
\section{Local solutions in Lagrangian coordinates}
\label{Loc}
\subsection{Lagrangian dynamics}
We introduce the following flow map for the two-component system \eqref{two-comp} (cf.\,\,\cite{YQZ15, WY23}):
\begin{equation}\label{y-iv}
\begin{split}
&\pt y(t,\xi)=
\left(\px u-u\right)(\px v+v)(t,y(t,\xi)),
\quad \xi\in\mathbb{R}.
\end{split}
\end{equation}
It turns out that using \eqref{y-iv}, we are able to obtain an explicit 
representation for the solution $(u,v)(t,x)$ in terms of $y(t,\xi)$ and the 
known initial data. All the derivations below are formal and will be justified later.
Form \eqref{two-comp} we have
\begin{equation}\label{cons-0}
\frac{d}{dt}\left[m(t,y(t,\xi))\partial_\xi y(t,\xi)\right]=0\quad\mbox{and}\quad
\frac{d}{dt}\left[n(t,y(t,\xi))\partial_\xi y(t,\xi)\right]=0,
\end{equation}
which imply
\begin{subequations}
	\label{cons-uv}
	\begin{align}
	\label{cons-u}
	&m(t,y(t,\xi))\partial_\xi y(t,\xi)=m_0(y_0(\xi))\pxi y_0(\xi),\\
	\label{cons-v}
	&n(t,y(t,\xi))\partial_\xi y(t,\xi)=n_0(y_0(\xi))\pxi y_0(\xi),
	\end{align}
\end{subequations}
where $m_0(x)=m(0,x)$, $n_0(x)=n(0,x)$ and $y_0(\xi)=y(0,\xi)$.

Assume that $y(t,\xi)\to\pm\infty$ as $\xi\to\pm\infty$ and $y(t,\xi)$ 
is strictly monotone increasing in $\xi$ for all fixed $t$.
Taking into account that $u(t,x)=\frac{1}{2}e^{-|x|}\ast m(t,x)$ and 
using \eqref{cons-u}, we can obtain the following formula 
for the component $u(t,x)$
(cf.\,\,\cite[equation (8)]{GL18}):
\begin{equation}
\nonumber
\begin{split}
u(t,x)&=\frac{1}{2}\int_{-\infty}^{\infty}
e^{-|x-z|}m(t,z)\,dz=
\frac{1}{2}\int_{-\infty}^{\infty}
e^{-|x-y(t,\xi)|}m(t,y(t,\xi))\pxi y(t,\xi)\,d\xi.
\\
&=
\frac{1}{2}\int_{-\infty}^{\infty}
e^{-|x-y(t,\xi)|}m_0(y_0(\xi))\pxi y_0(\xi)\,d\xi.
\end{split}
\end{equation}
Arguing similarly for $v(t,x)$, we obtain the following representations for the components $u$ and $v$:
\begin{subequations}\label{uv-repr}
\begin{align}
\label{u-repr}
&u(t,x)=\frac{1}{2}\int_{-\infty}^{\infty}
e^{-|x-y(t,\eta)|}m_0(y_0(\eta))\peta
y_0(\eta)\,d\eta,\\
\label{v-repr}
&v(t,x)=
\frac{1}{2}\int_{-\infty}^{\infty}
e^{-|x-y(t,\eta)|}n_0(y_0(\eta))
\peta y_0(\eta)\,d\eta.
\end{align}
\end{subequations}
Then \eqref{uv-repr} imply that
\begin{equation}
\label{uv_x-repr}
\begin{split}
&\px u(t,x)=-\frac{1}{2}
\int_{-\infty}^{\infty}
\sgn{x-y(t,\eta)}e^{-|x-y(t,\eta)|}
m_0(y_0(\eta))\peta y_0(\eta)\,d\eta,\\
&\px v(t,x)=-\frac{1}{2}
\int_{-\infty}^{\infty}
\sgn{x-y(t,\eta)}e^{-|x-y(t,\eta)|}
n_0(y_0(\eta))\pxi y_0(\eta)\,d\eta.
\end{split}
\end{equation}

The equations \eqref{uv-repr} and \eqref{uv_x-repr} lead us to 
consider \eqref{y-iv}, subject to the initial condition $y_0(\xi)$, 
as a Cauchy problem for an ordinary differential equation in a 
Banach space. Here, the ODE vector field is characterized by 
the known data $m_0(x)$ and $n_0(x)$. To properly 
formulate this problem, we proceed to define
\begin{equation}
	\label{UV}
	\begin{split}
	&U(t,\xi)\equiv u(t,y(t,\xi))=\frac{1}{2}
	\int_{-\infty}^{\infty}
	e^{-|y(t,\xi)-y(t,\eta)|}
	m_0(y_0(\eta))\peta y_0(\eta)\,d\eta,\\
	&V(t,\xi)\equiv v(t,y(t,\xi))=\frac{1}{2}
	\int_{-\infty}^{\infty}
	e^{-|y(t,\xi)-y(t,\eta)|}
	n_0(y_0(\eta))\peta y_0(\eta)\,d\eta,
	\end{split}
\end{equation}
and
\begin{equation}
	\label{WZ}
	\begin{split}
	&W(t,\xi)\equiv \px u(t,y(t,\xi))\\
	&\qquad\quad\,=-
	\frac{1}{2}
	\int_{-\infty}^{\infty}
	\sgn{y(t,\xi)-y(t,\eta)}
	e^{-|y(t,\xi)-y(t,\eta)|}
	m_0(y_0(\eta))\peta y_0(\eta)\,d\eta,\\
	&Z(t,\xi)\equiv \px v(t,y(t,\xi))\\
	&\qquad\quad\,=-\frac{1}{2}
	\int_{-\infty}^{\infty}
	\sgn{y(t,\xi)-y(t,\eta)}
	e^{-|y(t,\xi)-y(t,\eta)|}
	n_0(y_0(\eta))\peta y_0(\eta)\,d\eta.
	\end{split}
\end{equation}
Observe that when $m_0(x)$ and $n_0(x)$ belong to the 
space $L^1(\mathbb{R})$, we can establish the following  
uniform estimate with respect to $t$ for the functions 
$U$, $W$, $V$, and $Z$:
\begin{equation}\label{UZ-L1}
|U(t,\xi)|, |W(t,\xi)|
\leq\frac{1}{2}\|m_0\|_{L^1(\mathbb{R})},
\quad |V(t,\xi)|, |Z(t,\xi)|
\leq\frac{1}{2}\|n_0\|_{L^1(\mathbb{R})},
\end{equation}
assuming only that $y_0(\xi)$ exhibits monotonic 
behavior and tends towards $\pm \infty$ as $\xi\to \pm \infty$. 
Also it is convenient for us to introduce the function 
\begin{equation}\label{zeta}
\zeta(t,\xi)=y(t,\xi)-\xi,
\end{equation}
which will turn out to be bounded for $\xi\in\mathbb{R}$, 
see \eqref{s-zeta} below.

Taking into account \eqref{y-iv} and that $\pt\zeta(t,\xi)=\pt y(t,\xi)$, 
we obtain the following Cauchy problem:
\begin{equation}\label{zeta-iv}
\begin{split}
&\pt\zeta(t,\xi)=(W-U)(Z+V)(t,\xi),\\
&\zeta(0,\xi)=y_0(\xi)-\xi,
\end{split}
\end{equation}
which is considered in the Banach space 
$E_l\subset C^{1}(\mathbb{R})$ defined as follows 
(cf.\,\,\cite[Section 2.2]{HR07} and \cite[equation (16)]{GL18})
\begin{equation}\label{E-l}
E_l=\{f(\xi)\in C^{1}(\mathbb{R})
\,\,|\,\,\pxi f(\xi)\geq l-1,\,\, \mbox{ for all } \xi\in\mathbb{R} \},\quad l\geq 0,
\end{equation}
with the norm 
$\|f\|_{E_l}=\|f\|_{C^{1}(\mathbb{R})}$. 
Notice that if $\zeta(t,\cdot)\in E_l$ with $l>0$, then 
$\pxi y(t,\xi)\geq l$ for all $\xi$ and therefore $y(t,\cdot)$ 
is strictly monotone increasing.

\subsection{Local characteristic}
Introduce the following operator, corresponding to \eqref{zeta-iv}:
\begin{equation}\label{Ay-def}
A(\zeta)(t,\xi)=
y_0(\xi)-\xi+\int_0^t(W-U)(Z+V)
(\tau,\xi)\,d\tau,\quad t\in[-T,T].
\end{equation}
In Proposition \ref{A-contr-prop} below, we will prove that $A$ is a contraction in 
$C([-T,T],E_l)$ for a class of initial data $y_0(\xi)$ and sufficiently small $T>0$.
To demonstrate this, we establish the following technical lemma:
\begin{lemma}\label{U-diff}
	Suppose that $m_0(x),n_0(x)\in X^0$,
	$y(t,\cdot)$ is strictly monotone increasing for all $t\in[-\tilde T,T]$
	and
	$\zeta(t,\xi)\in 
	C\left([-\tilde T,T], 
	C^1(\mathbb{R})\right)$,
	for some $\tilde T,T>0$.
	Then we have
	\begin{enumerate}
	
	\item $U,W,V,Z\in 
	C\left([-\tilde T,T], 
	C^1(\mathbb{R})\right)$;
	\item the partial derivatives of $U,W,V,Z$ in $\xi$ have the form
	\begin{equation}\label{UZ-diff}
	\begin{split}
	&\pxi
	\begin{pmatrix}
	U(t,\xi)\\
	W(t,\xi)
	\end{pmatrix}=
	\begin{pmatrix}
	0& \pxi y(t,\xi)\\
	\pxi y(t,\xi)&0
	\end{pmatrix}
	\begin{pmatrix}
	U(t,\xi)\\
	W(t,\xi)
	\end{pmatrix}
	-\begin{pmatrix}
	0\\
	m_0(y_0(\xi))\pxi y_0(\xi)
	\end{pmatrix},\\
	&\pxi
	\begin{pmatrix}
	V(t,\xi)\\
	Z(t,\xi)
	\end{pmatrix}=
	\begin{pmatrix}
	0& \pxi y(t,\xi)\\
	\pxi y(t,\xi)&0
	\end{pmatrix}
	\begin{pmatrix}
	V(t,\xi)\\
	Z(t,\xi)
	\end{pmatrix}
	-\begin{pmatrix}
	0\\
	n_0(y_0(\xi))\pxi y_0(\xi)
	\end{pmatrix}.
	\end{split}
	\end{equation}
	\end{enumerate}
	\begin{proof}
	Consider the integrals
	\begin{equation}\label{J-def}
	\begin{split}
	&J_1(t,\xi;m_0)=
	\int_{-\infty}^{\xi}
	e^{y(t,\eta)-y(t,\xi)}m_0(y_0(\eta))
	\peta y_0(\eta)\,d\eta,\\
	&J_2(t,\xi;m_0)=
	\int_{\xi}^{\infty}
	e^{y(t,\xi)-y(t,\eta)}m_0(y_0(\eta))
	\peta y_0(\eta)\,d\eta.
	\end{split}
	\end{equation}
	Then \eqref{UV} and \eqref{WZ} imply
	\begin{equation}\label{UZ-J}
	\begin{split}
	& U(t,\xi)=\frac{1}{2}(J_1+J_2)(t,\xi;m_0),
	\quad
	W(t,\xi)=-\frac{1}{2}(J_1-J_2)(t,\xi;m_0),\\
	& V(t,\xi)=\frac{1}{2}(J_1+J_2)(t,\xi;n_0),
	\quad
	Z(t,\xi)=-\frac{1}{2}(J_1-J_2)(t,\xi;n_0).
	\end{split}
	\end{equation}
	Differentiating \eqref{UZ-J} with respect to $\xi$, direct calculations show \eqref{UZ-diff} and thus we have item (2) of the lemma.
	
	Now let us prove item (1).
	First, we show that $J_1(t,\cdot;m_0)$ is continuous.
	Denoting 
	$\tilde{m}_0(\eta)=m_0(y_0(\eta))
	\peta y_0(\eta)$ and using \eqref{exp-ineq}, we have for any $\xi_1,\xi_2\in\mathbb{R}$, $\xi_1\leq\xi_2$
	(here we drop the arguments $t$, $m_0$)
	\begin{equation*}
	\begin{split}
	\left|J_1(\xi_2)-J_1(\xi_1)\right|
	&\leq
	\left|\int_{-\infty}^{\xi_2}
	e^{y(t,\eta)-y(t,\xi_2)}\tilde{m}_0(\eta)\,d\eta
	-\int_{-\infty}^{\xi_1}
	e^{y(t,\eta)-y(t,\xi_2)}\tilde{m}_0(\eta)\,d\eta
	\right|\\
	&\quad+\left|
	\int_{-\infty}^{\xi_1}
	e^{y(t,\eta)-y(t,\xi_2)}\tilde{m}_0(\eta)\,d\eta
	-
	\int_{-\infty}^{\xi_1}
	e^{y(t,\eta)-y(t,\xi_1)}\tilde{m}_0(\eta)\,d\eta
	\right|\\
	&\leq
	\|m_0\|_{C}\|\pxi y_0\|_{C}|\xi_1-\xi_2|+
	\|m_0\|_{L^1}\|\partial_{(\cdot)} y(t,\cdot)\|_C
	|\xi_1-\xi_2|.
	\end{split}
	\end{equation*}
	Since the condition $\xi_1\leq\xi_2$ does not restrict the 
	generality, we have that $J_1$ is continuous in $\xi$.
	Arguing similarly, we conclude that $J_2(t,\cdot;m_0)$ and 
	$J_j(t,\cdot;n_0)$, $j=1,2$ are continuous, which, together with the estimates
	$\left|J_j(m_0)\right|\leq\|m_0\|_{L^1}$
	and $\left|J_j(n_0)\right|\leq\|n_0\|_{L^1}$, imply 
	that $J_j(t,\cdot;m_0)$, $J_j(t,\cdot;n_0)$ belong to $C(\mathbb{R})$.
	
	Now let us prove that these functions belong to 
	$C\left([-\tilde T,T], C(\mathbb{R})\right)$.
	As above, we consider $J_1(t)=J_1(t,\xi;m_0)$ only, the 
	other integrals can be treated similarly. 
	For all $t_1,t_2\in[-\tilde T,T]$,
	\begin{equation*}
	|J_1(t_1)-J_1(t_2)|
	\leq 2\|m_0\|_{L^1}\|y(t_1,\cdot)-y(t_2,\cdot)\|_C	
	=
	2\|m_0\|_{L^1}
	\|\zeta(t_1,\cdot)-\zeta(t_2,\cdot)\|_C,
	\end{equation*}
	where we have used \eqref{exp-ineq}.
	Since $\zeta\in 
	C\left([-\tilde T,T],C^1\right)$, we have that 
	$J_j(t,\xi;m_0), J_j(t,\xi;n_0)\in
	C\left([-\tilde T,T],C(\mathbb{R})\right)$ and \eqref{UZ-J} 
	implies that $U,W,V,Z$ also belong to this Banach space. 
	Finally, using \eqref{UZ-diff} and that
	$\px y\in C\left([-\tilde{T},T],C\right)$, we arrive at item 1 of the lemma.
	\end{proof}
\end{lemma}

Now we are at the position to prove the contraction of the operator $A$ defined by \eqref{Ay-def}.
\begin{proposition}\label{A-contr-prop}
	Consider $m_0(x),n_0(x)\in X^0$.
	Assume that 
	$(y_0(\xi)-\xi)\in E_c$ for some $c>0$.
	Take any $0<l<c$.
	Then for all
	\begin{equation}\label{T}
	0<T<\min\left\{
	\frac{1}{4\|m_0\|_{L^{1}}\|n_0\|_{L^{1}}},
	\frac{\min\{1/2,c-l\}}
	{\|\pxi y_0\|_{C}
		\left(\|m_0\|_{C}\|n_0\|_{L^1}+
		\|m_0\|_{L^1}\|n_0\|_{C}
		\right)}
	\right\},
	\end{equation}
	the operator $A$ defined in \eqref{Ay-def} is a contraction in $C([-T,T],E_l)$:
	\begin{equation}\label{A-contr}
	\begin{split}
	\|A(\zeta_1)(t,\xi)-A(\zeta_2)(t,\xi)\|
	_{C([-T,T],E_l)}
	&\leq\alpha\|\zeta_1(t,\xi)-\zeta_2(t,\xi)\|
	_{C([-T,T],C)}\\
	&\leq\alpha\|\zeta_1(t,\xi)-\zeta_2(t,\xi)\|
	_{C([-T,T],E_l)},
	\end{split}
	\end{equation}
	for all $\zeta_j(t,\xi)\in C([-T,T],E_l)$, $j=1,2$ and some 
	$0<\alpha<1$, $\alpha=\alpha(T)$.
\end{proposition}
\begin{proof}
	Firstly, note that a function $y_0$ meeting the conditions of the 
	proposition does indeed exist. This can be achieved by 
	choosing $y_0(\xi)=c\xi$, where $c$ is a positive constant.

	Let us prove that 
	$A(\zeta)(t,\xi)\in 
	C\left([-T,T],E_l\right)$ for 
	any $\zeta(t,\xi)\in C([-T,T],E_l)$, where $E_l$ is defined in \eqref{E-l}.
	Differentiating \eqref{Ay-def} in $\xi$ and applying \eqref{UZ-diff}, we obtain 
	\begin{equation}\label{pdpsiA}
	\begin{split}
	\pxi(A(\zeta))(t,\xi)=\pxi y_0(\xi)-1
	-\pxi y_0(\xi)&\left(
	m_0(y_0(\xi))\int_0^t(Z+V)(\tau,\xi)\,d\tau\right.\\
	&\left.
	\quad+n_0(y_0(\xi))\int_0^t(W-U)(\tau,\xi)\,d\tau
	\right).
	\end{split}
	\end{equation}
	Combining \eqref{Ay-def}, \eqref{pdpsiA} and item (1) in Lemma \ref{U-diff}, we conclude that
	$A(\zeta)(t,\xi)\in C\left([-T,T],C^1(\mathbb{R})\right)$.
	
	It remains to prove that 
	$\pxi A(\zeta)(t,\xi)\geq l-1$ 
	for all $t,\xi\in[-T,T]\times\mathbb{R}$.
	From \eqref{pdpsiA} and \eqref{UZ-L1} we have the following inequality:
	\begin{equation*}
	\begin{split}
	|\pxi(A(\zeta))(t,\xi)|&\geq
	c-1-T\|\pxi y_0\|_{C}
	\left(\|m_0\|_{C}(\|V\|_C+\|Z\|_C)+
	\|n_0\|_{C}(\|U\|_C+\|W\|_C)
	\right)\\
	&\geq c-1-T\|\pxi y_0\|_{C}
	\left(\|m_0\|_{C}\|n_0\|_{L^1}+
	\|m_0\|_{L^1}\|n_0\|_{C}
	\right),
	\end{split}
	\end{equation*}
	for all $t,\xi\in[-T,T]\times\mathbb{R}$.
	Taking $T$ as in \eqref{T}, we conclude that 
	$\pxi A(\zeta)(t,\xi)\geq l-1$.
	
	Now let us prove \eqref{A-contr}.
	Let $U_j$, $W_j$, $V_j$ and $Z_j$ $j=1,2$, denote $U$, 
	$W$, $V$ and $Z$ respectively with 
	$\zeta_j(t,\xi)=y_j(t,\xi)-\xi$ instead of $\zeta(t,\xi)=y(t,\xi)-\xi$.
	Using that $y(t,\cdot)$ is strictly monotone increasing, the 
	inequality \eqref{exp-ineq} and that $||a|-|b||\leq|a-b|$ 
	for all $a,b\in\mathbb{R}$, we have
	\begin{subequations}\label{U_1-U_2}
	\begin{equation}
	\begin{split}
	|U_{1}(t,\xi)-&U_{2}(t,\xi)|\leq
	\frac{1}{2}\int_{-\infty}^\infty
	\left||y_2(t,\eta)-y_2(t,\xi)|
	-|y_1(t,\eta)-y_1(t,\xi)|\right|
	|m_0(y_0(\eta))\peta y_0(\eta)|\,d\eta\\
	&\leq\|m_0\|_{L^{1}}
	\|y_1(t,\cdot)-y_2(t,\cdot)\|_{C}=
	\|m_0\|_{L^{1}}
	\|\zeta_1(t,\cdot)-\zeta_2(t,\cdot)\|
	_{C},
	\end{split}
	\end{equation}
	for all $t,\xi\in[-T,T]\times\mathbb{R}$.
	Arguing similarly, we obtain
	\begin{align}
	\label{W_1-W_2}
	&|W_{1}(t,\xi)-W_{2}(t,\xi)|\leq
	\|m_0\|_{L^{1}}
	\|\zeta_1(t,\cdot)-\zeta_2(t,\cdot)\|_{C},\\
	&|V_{1}(t,\xi)-V_{2}(t,\xi)|\leq
	\|n_0\|_{L^{1}}
	\|\zeta_1(t,\cdot)-\zeta_2(t,\cdot)\|_{C},\\
	&|Z_{1}(t,\xi)-Z_{2}(t,\xi)|\leq
	\|n_0\|_{L^{1}}
	\|\zeta_1(t,\cdot)-\zeta_2(t,\cdot)\|_{C}.
	\end{align}
	\end{subequations}
	Using Item (1) of Lemma \ref{U-diff}, \eqref{U_1-U_2} and \eqref{ab-diff}
	it implies from \eqref{Ay-def} that (we drop the 
	arguments $t,\xi$ for $A(\zeta_j)$, $j=1,2$)
	\begin{equation}\label{A-diff}
	\begin{split}
	\left|A(\zeta_1)-A(\zeta_2)\right|
	&\leq T\|n_0\|_{L^{1}}(|U_1-U_2|+|W_1-W_2|)
	+T\|m_0\|_{L^{1}}(|V_1-V_2|+|Z_1-Z_2|)\\
	&\leq 4T\|m_0\|_{L^{1}}\|n_0\|_{L^{1}}
	\|\zeta_1(t,\cdot)-\zeta_2(t,\cdot)\|_{C}.
	\end{split}
	\end{equation}
	To estimate $|\pxi A(\zeta_1)-\pxi A(\zeta_2)|$, 
	we use \eqref{pdpsiA} together with \eqref{U_1-U_2}, which imply
	\begin{equation}\label{pxiA-diff}
	|\pxi A(\zeta_1)-\pxi A(\zeta_2)|
	\leq 2T	\|\pxi y_0\|_{C}
	(\|m_0\|_{C}\|n_0\|_{L^1}+
	\|m_0\|_{L^1}\|n_0\|_{C})
	\|\zeta_1(t,\cdot)-\zeta_2(t,\cdot)\|_{C}.
	\end{equation}
	Combining \eqref{A-diff} and \eqref{pxiA-diff} with $T$ satisfying \eqref{T}, we arrive at \eqref{A-contr}.
\end{proof}

Using Proposition \ref{A-contr-prop}, we can easily prove that there exists a unique local solution of the Cauchy problem \eqref{zeta-iv} in the Banach space $C([-T,T],E_l)$, see Proposition \ref{char-loc} below.
However, to establish the decay rate of the 
solution $\zeta(t,\cdot)$, we will need the following lemma.
\begin{lemma}\label{U-L^1}
	Suppose that 
	$m_0(x),n_0(x)\in X^0$,
	$\zeta(t,\xi)\in
	C^1\left([-T,T], C(\mathbb{R})\right)$
	for some $T>0$ and $y_0(\xi)\in E_l$, $l>0$.
	Then $U,W,V,Z\in C\left([-T,T],L^1(\mathbb{R})\right)$.
\end{lemma}
\begin{proof}
	In view of \eqref{UZ-J}, it is enough to show that 
	the integrals $J_j(m_0)$ and $J_j(n_0)$, $j=1,2$, defined in \eqref{J-def} belong to 
	$C\left([-T,T],L^1(\mathbb{R})\right)$.
	We give a proof for $J_1(t,\xi;m_0)$, the other integrals can be treated similarly.
	
	Changing the order of integration and using that $\|\zeta(t,\cdot)\|_{C(\mathbb{R})}
	=\|y(t,\cdot)-(\cdot)\|_{C(\mathbb{R})}$
	is finite for all $t\in[-T,T]$, we have
	\begin{equation}\label{J-1-L^1}
	\begin{split}
	\|J_1(t,\cdot;m_0)\|_{L^{1}}&\leq
	\int_{-\infty}^{\infty}
	\int_\eta^\infty e^{-y(t,\xi)}\,d\xi\,
	e^{y(t,\eta)}|m_0(y_0(\eta))\peta y_0(\eta)|\,d\eta\\
	&\leq e^{\|\zeta(t,\cdot)\|_{C}}
	\int_{-\infty}^{\infty}
	e^{y(t,\eta)-\eta}
	|m_0(y_0(\eta))\peta y_0(\eta)|\,d\eta
	\leq e^{2\|\zeta(t,\cdot)\|_{C}}
	\|m_0\|_{L^1}.
	\end{split}
	\end{equation}
	Now let us establish the continuity of the map 
	$t\mapsto \|J_1(t)\|_{L^1}$.
	Changing the order of integration as in \eqref{J-1-L^1}, 
	we have for any $t_1,t_2\in[-T,T]$
	\begin{equation*}
	\begin{split}
	\|J_1(t_1)-J_1(t_2)\|_{L^1}&\leq
	\int_{-\infty}^{\infty}\int_{\eta}^\infty
	e^{y(t_1,\eta)}\left|
	e^{-y(t_1,\xi)}-e^{-y(t_2,\xi)}
	\right|\,d\xi
	\, |m_0(y_0(\eta))\peta y_0(\eta)|\,d\eta\\
	&\quad+\int_{-\infty}^{\infty}\int_{\eta}^\infty
	e^{-y(t_2,\xi)}\left|
	e^{y(t_1,\eta)}-e^{y(t_2,\eta)}
	\right|\,d\xi
	\, |m_0(y_0(\eta))\peta y_0(\eta)|\,d\eta\\
	&=I_1+I_2.
	\end{split}
	\end{equation*}
	The integral $I_1$ can be estimated by 
	using the mean value theorem and taking into account that
	$\pt y=\pt\zeta$, see \eqref{zeta}, as follows:
	\begin{equation}\label{I_1-in}
	\begin{split}
	I_1&\leq\int_{-\infty}^{\infty}
	e^{y(t_1,\eta)}\int_{\eta}^\infty
	|\pt y(t^*,\xi)|e^{-y(t^*,\xi)}|t_1-t_2|\,d\xi
	|m_0(y_0(\eta))\peta y_0(\eta)|\,d\eta\\
	&\leq\|\pt\zeta(t^*,\cdot)\|_C
	e^{\|\zeta(t_1,\cdot)\|_C+\|\zeta(t^*,\cdot)\|_C}
	\|m_0\|_{L^1}|t_1-t_2|,
	\end{split}
	\end{equation}
	for some $t^*$ between $t_1$ and $t_2$.
	In a similar manner we can estimate $I_2$ and thus eventually conclude that
	$J_j(m_0),J_j(n_0)\in 
	C\left([-T,T],L^1\right)$,
	$j=1,2$.
\end{proof}
Now we can show that there exists a unique local solution $\zeta(t,\xi)$  of the Cauchy problem \eqref{zeta-iv}, which has additional regularity and decay rate for a class of initial data $m_0,n_0$.

\begin{proposition}[Existence and uniqueness of the local characteristics]
	\label{char-loc}
	Suppose that $m_0(x)$, 
	$n_0(x)\in X^k$,
	$(y_0(\xi)-\xi)\in X^{k+1}$ for some $k\in\mathbb{N}_0$ and
	$(y_0(\xi)-\xi)\in E_c$, $c>0$.
	Then for any $0<l<c$ and $T$ satisfying \eqref{T} there exists a unique 
	$\zeta(t,\xi)\in C([-T,T],E_l)$ such that
	\begin{equation}\label{zeta-int}
	\zeta(t,\xi)=
	y_0(\xi)-\xi+\int_0^t(W-U)(Z+V)
	(\tau,\xi)\,d\tau,\quad t\in[-T,T],
	\end{equation}
	which is a unique local solution of the Cauchy problem \eqref{zeta-iv} in the Banach space $C([-T,T],E_l)$.
	Moreover, the solution $\zeta(t,\xi)$ has the following regularity and decay properties:
	\begin{equation}\label{zeta-reg}
	\zeta(t,\xi)\in
	C^1\left([-T,T],
	X^{k+1}
	\right).
	\end{equation}
	
	Finally, $\zeta(t,\xi)$ satisfies the following size estimates:
	\begin{subequations}\label{s-zeta}
	\begin{align}
	\label{s-zeta-a}
	&\|\zeta(t,\cdot)\|_{C}\leq
	\|y_0(\cdot)-(\cdot)\|_{C}
	+T
	\|m_{0}\|_{L^{1}}
	\|n_{0}\|_{L^{1}},\\
	\label{s-zeta-c}
	&\|\partial_{(\cdot)}^{j+1}\zeta(t,\cdot)\|_{C}\leq
	1+\|\pxi^{j+1} y_0\|_{C}+
	TC_j,\quad j=0,\dots,k,
	\end{align}
	\end{subequations}
	with 
	$$C_0=\|\pxi y_0\|_{C}
	\left(
	\|m_0\|_{C}\|n_0\|_{L^{1}}
	+\|m_0\|_{L^1}\|n_0\|_{C}
	\right),$$ and
	some $C_j=C_j(\|\pxi^j y_0\|_C,\|m_0\|_{X^{j}},
	\|n_0\|_{X^{j}})>0$.
\end{proposition}
\begin{proof}
	The existence and uniqueness of the fixed point of the 
	operator $A$ in the Banach space $C([-T,T], E_l)$ 
	follows from Proposition \ref{A-contr-prop}
	and the contraction mapping theorem.
	
	Let us prove that $\zeta\in C([-T,T], C^{k+1})$ by induction.
	We already know (see \eqref{E-l}) that
	$\zeta\in C\left([-T,T], C^{1}\right)$
	and thus the base case of the induction is established.
	Suppose that 
	$\zeta\in C\left([-T,T], C^{j}\right)$, 
	for some $j=1,\dots,k$.
	To show that 
	$\zeta\in C\left([-T,T], C^{j+1}\right)$
	we notice that \eqref{pdpsiA} implies 
	\begin{equation}\label{pdpsi-zeta}
	\begin{split}
	\pxi\zeta(t,\xi)=\pxi y_0(\xi)-1
	-\pxi y_0(\xi)&\left(
	m_0(y_0(\xi))\int_0^t(Z+V)(\tau,\xi)\,d\tau\right.\\
	&\left.
	\quad+n_0(y_0(\xi))\int_0^t(W-U)(\tau,\xi)\,d\tau
	\right),
	\end{split}
	\end{equation}
	where the right-hand side does not depend on $\pxi\zeta(t,\xi)$.
	Therefore \eqref{pdpsi-zeta} and item (2) of Lemma \ref{U-diff} imply that 
	\begin{equation}\label{zeta-diff-j}
	\pxi^{j+1}\zeta(t,\xi)=
	\pxi^{j+1}(y_0(\xi)-\xi)
	+\mathcal{P}_1\int_0^t\mathcal{P}_2\,d\tau
	+\mathcal{P}_3\int_0^t\mathcal{P}_4\,d\tau,
	\end{equation}
	where $\mathcal{P}_r$, $r=1,\dots,4$,
	are polynomials which depend on 
	$U,W,V,Z$, $\{\pxi^i\zeta\}_{i=0}^{j}$, 
	$\{\pxi^im_0\}_{i=0}^{j}$, $\{\pxi^in_0\}_{i=0}^{j}$ and $\{\pxi^i(y_0(\xi)-\xi)\}_{i=1}^{j+1}$.
	Using item (1) of Lemma \ref{U-diff} as well as the induction hypothesis, we conclude that
	$\zeta\in C\left([-T,T],C^{j+1}\right)$,
	which completes the induction step.
	Therefore we have established that 
	$\zeta\in C\left([-T,T],C^{k+1}\right)$.
	Arguing similarly for $\pt\zeta$ and $\pxi\pt\zeta$ we conclude that
	$\pt\zeta\in C\left([-T,T],C^{k+1}\right)$
	and therefore 
	$\zeta\in C^1\left([-T,T],C^{k+1}\right)$.
	
	We also establish that $\zeta\in C\left([-T,T],W^{k+1,1}\right)$ 
	through an inductive approach.
	Lemma \ref{U-L^1}, \eqref{zeta-int}, \eqref{pdpsi-zeta}, 
	and \eqref{UZ-L1} imply that 
	$\zeta\in C\left([-T,T], W^{1,1}\right)$.
	The induction step can be proved by using 
	\eqref{zeta-diff-j} and applying again Lemma \ref{U-L^1}, 
	item (1) of Lemma \ref{U-diff} and the induction hypothesis, that is 
	$\zeta\in 
	C\left([-T,T], W^{j,1}\right)$, 
	$j=1,\dots k$.
	Reasoning in the similar manner, we can prove that
	$\pt\zeta\in
	C\left([-T,T], W^{k+1,1}\right)$.
	
	Finally, inequalities \eqref{s-zeta-a} and, \eqref{s-zeta-c} with $j=0$ follow from \eqref{UZ-L1}, \eqref{zeta-int} and \eqref{pdpsi-zeta}, while
	\eqref{s-zeta-c} for $j=1,\dots,k$ follows from \eqref{zeta-diff-j}.
\end{proof}
\begin{remark}
	Notice that the time $T>0$ in \eqref{T} depends on the initial data $m_0,n_0$ and it cannot be extended by choosing certain specific initial value $y_0$ of the Cauchy problem \eqref{zeta-iv}.
\end{remark}

Finally, let us prove the 
continuous dependence of $\zeta$ on $m_0,n_0$, which
will be used in Section \ref{lwps} for proving the Lipschitz continuity properties of the solution $(m,n)$, see Corollary \ref{Lip-c} below.
\begin{proposition}[Lipschitz continuity of $\zeta$ on $(m_0,n_0)$]
	\label{Lip-cont-ch}
	Fix any two constants $0<R_0\leq R$.
	Suppose that $m_{0,j}(x),n_{0,j}(x)\in X^{k}$, $j=1,2$, for some $k\in\mathbb{N}_0$, are such that
	$$
	\|m_{0,j}\|_{X^{0}},
	\|n_{0,j}\|_{X^{0}}\leq R_0,\quad
	\mbox{and}\quad
	\|m_{0,j}\|_{X^{k}},
	\|n_{0,j}\|_{X^{k}}\leq R,\quad j=1,2.
	$$
	Also assume that
	$(y_0(\xi)-\xi)\in X^{k+1}$ and
	$(y_0(\xi)-\xi)\in E_c$ for some $c>0$.
	Consider the corresponding characteristics (see Proposition \ref{char-loc})
	\begin{equation}\label{zeta-j-int}
	\zeta_j(t,\xi)=
	y_0(\xi)-\xi+\int_0^t
	\left(\hat{W}_j-\hat{U}_j\right)
	\left(\hat{Z}_j+\hat{V}_j\right)
	(\tau,\xi)\,d\tau,\quad t\in[-T,T],\quad j=1,2,
	\end{equation}
	where $\hat{U}_j$, $\hat{W}_j$, $\hat{V}_j$ and
	$\hat{Z}_j$ are defined by \eqref{UV}, \eqref{WZ} with
	$\zeta_j=y_j-\xi$, $m_{0,j}$ and $n_{0,j}$ instead of
	$y$, $m_0$ and $n_0$ respectively, $j=1,2$. 
	Then we have the following Lipschitz property of 
	$\zeta_j$ for a sufficiently small $T>0$ (here we drop the 
	arguments of functions for simplicity):
	\begin{subequations}\label{Lip-cont-zeta}
	\begin{align}
	\label{Lip-cont-zeta-a}
	&\|\zeta_1-\zeta_2\|
	_{C\left([-T,T],X^0\right)}
	\leq C_1\left(\|m_{0,1}-m_{0,2}\|_{L^{1}}
	+\|n_{0,1}-n_{0,2}\|_{L^{1}}
	\right),\\
	&\label{Lip-cont-zeta-b}
	\|\zeta_1-\zeta_2\|
	_{C\left([-T,T],X^{r+1}\right)}
	\leq C_{2,r}\left(\|m_{0,1}-m_{0,2}\|_{X^{r}}
	+\|n_{0,1}-n_{0,2}\|_{X^{r}}
	\right),\quad r=0,\dots,k,
	\end{align}
	\end{subequations}
	for some $C_1=C_1(T,\|y_0(\cdot)-(\cdot)\|_{C^1},R_0)>0$ and
	$C_{2,r}=
	C_{2,r}(T,\|y_0(\cdot)-(\cdot)\|_{C^{r+1}},R)>0$.
\end{proposition}
\begin{proof}
	Introduce the following integrals (cf.\,\,\eqref{J-def}):
	\begin{equation*}
	\begin{split}
	&\hat{J}_{1,j}(t,\xi;m_{0,j})=
	\int_{-\infty}^{\xi}
	e^{y_j(t,\eta)-y_j(t,\xi)}m_{0,j}(y_0(\eta))
	\peta y_0(\eta)\,d\eta,\quad j=1,2,\\
	&\hat{J}_{2,j}(t,\xi;m_{0,j})=
	\int_{\xi}^{\infty}
	e^{y_j(t,\xi)-y_j(t,\eta)}m_{0,j}(y_0(\eta))
	\peta y_0(\eta)\,d\eta,\quad j=1,2.
	\end{split}
	\end{equation*}
	Then we have
	\begin{equation}\label{UZ-hatJ}
	\begin{split}
	& \hat{U}_j(t,\xi)=
	\frac{1}{2}\left(\hat{J}_{1,j}+\hat{J}_{2,j}\right)
	(t,\xi;m_{0,j}),
	\quad
	\hat{W}_j(t,\xi)=
	-\frac{1}{2}\left(\hat{J}_{1,j}-\hat{J}_{2,j}\right)
	(t,\xi;m_{0,j}),\\
	& \hat{V}_j(t,\xi)=
	\frac{1}{2}\left(\hat{J}_{1,j}+\hat{J}_{2,j}\right)
	(t,\xi;n_{0,j}),
	\quad
	\hat{Z}_j(t,\xi)=
	-\frac{1}{2}\left(\hat{J}_{1,j}-\hat{J}_{2,j}\right)
	(t,\xi;n_{0,j}),
	\end{split}
	\end{equation}
	where $j=1,2$.
	
	First, we prove \eqref{Lip-cont-zeta-a}.
	Observe that
	(here 
	$\hat{J}_{1,j}(m_{0,j})
	=\hat{J}_{1,j}(t,\xi;m_{0,j})$, $j=1,2$)
	\begin{equation*}
	\begin{split}
	|\hat{J}_{1,1}(m_{0,1})-\hat{J}_{1,2}(m_{0,2})|&\leq
	\int_{-\infty}^{\xi}\left|
	e^{y_1(t,\eta)-y_1(t,\xi)}-e^{y_2(t,\eta)-y_2(t,\xi)}
	\right||m_{0,1}(y_0(\eta))\peta y_0(\eta)|\,d\eta\\
	&\quad +\int_{-\infty}^{\xi}
	e^{y_2(t,\eta)-y_2(t,\xi)}
	|(m_{0,1}-m_{0,2})(y_0(\eta))\peta y_0(\eta)|
	\,d\eta\\
	&\leq 2\|m_{0,1}\|_{L^1}
	\|\zeta_1(t,\cdot)-\zeta_2(t,\cdot)\|_{C}
	+\|m_{0,1}-m_{0,2}\|_{L^1},
	\end{split}
	\end{equation*}
	where in the second inequality we have used \eqref{exp-ineq}.
	Arguing similarly for 
	$|\hat{J}_{2,1}(m_{0,1})-\hat{J}_{2,2}(m_{0,2})|$
	and
	$|\hat{J}_{i,1}(n_{0,1})-\hat{J}_{i,2}(n_{0,2})|$, $i=1,2$, 
	we conclude from \eqref{UZ-hatJ} (we drop the arguments for simplicity)
	\begin{equation}\label{hatU-diff}
	\begin{split}
	|\hat{U}_1-\hat{U}_2|,|\hat{W}_1-\hat{W}_2|,
	|\hat{V}_1-\hat{V}_2|,|\hat{Z}_1-\hat{Z}_2|&\leq
	2\left(\|m_{0,1}\|_{L^1}+\|n_{0,1}\|_{L^1}\right)
	\|\zeta_1(t,\cdot)-\zeta_2(t,\cdot)\|_{C}\\
	&\quad+\|m_{0,1}-m_{0,2}\|_{L^1}
	+\|n_{0,1}-n_{0,2}\|_{L^1}.
	\end{split}
	\end{equation}
	Then combining \eqref{zeta-j-int} and \eqref{ab-diff}, \eqref{UZ-L1} for $U_j$, $W_j$, $V_j$, $Z_j$, $j=1,2$, we arrive at
	(here $C=C\left([-T,T],C\right)$)
	\begin{equation}\label{z1-z2}
	\|\zeta_1-\zeta_2\|_{C}
	\leq T\tilde{C}\left(\|U_1-U_2\|_C+
	\|W_1-W_2\|_C+\|V_1-V_2\|_C+\|Z_1-Z_2\|_C\right),
	\end{equation}
	for some 
	$\tilde{C}=
	\tilde{C}(\|m_{0,j}\|_{L^1},\|n_{0,j}\|_{L^1})$.
	Inequality \eqref{z1-z2} together with \eqref{hatU-diff} implies \eqref{Lip-cont-zeta-a}, 
	for a sufficiently small $T>0$, with the norm $\|\cdot\|_{C([-T,T],C)}$ 
	utilized on the left hand side.
	
	To obtain \eqref{Lip-cont-zeta-a} for the norm 
	$\|\cdot\|_{C([-T,T],L^1)}$, we should estimate the $L^1$-norms 
	of the differences on the left hand side of \eqref{hatU-diff}.
	Notice that
	\begin{equation*}
	\int_{-\infty}^{\infty}
	|\hat{J}_{1,1}(m_{0,1})-\hat{J}_{1,2}(m_{0,2})|\,d\xi
	\leq I_3+I_4,
	\end{equation*}
	where (here we change the order of integration and use the notation 
	$\tilde{M}_{0,1}(\eta)=|m_{0,1}(y_0(\eta))\peta y_0(\eta)|$)
	\begin{equation*}
	\begin{split}
	&I_3=\int_{-\infty}^{\infty}
	\tilde{M}_{0,1}(\eta)\int_{\eta}^{\infty}\left|
	e^{y_1(t,\eta)-y_1(t,\xi)}-e^{y_2(t,\eta)-y_2(t,\xi)}
	\right|\,d\xi\, d\eta,\\
	&I_4=\int_{-\infty}^{\infty}
	|(m_{0,1}-m_{0,2})(y_0(\eta))\peta y_0(\eta)|
	\int_{\eta}^{\infty}
	e^{y_2(t,\eta)-y_2(t,\xi)}
	\,d\xi \, d\eta.
	\end{split}
	\end{equation*}
	Recalling that $y_j(t,\xi)=\zeta_j(t,\xi)-\xi$ and 
	using \eqref{exp-ineq-b}, we obtain
	\begin{equation*}
	\begin{split}
	I_3&\leq\int_{-\infty}^{\infty}
	\tilde{M}_{0,1}(\eta)\int_{\eta}^{\infty}
	e^{y_1(t,\eta)}
	\left|e^{-y_1(t,\xi)}-e^{-y_2(t,\xi)}\right|
	+e^{-y_2(t,\xi)}
	\left|e^{y_1(t,\eta)}-e^{y_2(t,\eta)}\right|
	\,d\xi\,d\eta\\
	&\leq\int_{-\infty}^{\infty}
	\tilde{M}_{0,1}(\eta)e^{y_1(t,\eta)}
	\int_{\eta}^{\infty}
	e^{-\eta}
	\left|e^{-\zeta_1(t,\xi)}
	-e^{-\zeta_2(t,\xi)}\right|
	\,d\xi\,d\eta\\
	&\quad+
	e^{\|\zeta_2(t,\cdot)\|_C}\int_{-\infty}^{\infty}
	\tilde{M}_{0,1}(\eta)
	\left|e^{y_1(t,\eta)}-e^{y_2(t,\eta)}\right|
	\int_{\eta}^{\infty} e^{-\xi}\,d\xi\,d\eta\\
	&\leq (\|m_{0,1}\|_{L^1}
	+\|m_{0,1}\|_{C}\|\pxi y_0\|_C)
	e^{2\max
	\{\|\zeta_1(t,\cdot)\|_C,\|\zeta_2(t,\cdot)\|_C\}}
	\|\zeta_1(t,\cdot)-\zeta_2(t,\cdot)\|_{L^1}.
	\end{split}
	\end{equation*}
	The integral $I_4$ can be estimated as follows:
	\begin{equation*}
	\begin{split}
	I_4&\leq e^{\|\zeta_2(t,\cdot)\|_C}
	\int_{-\infty}^{\infty}
	|(m_{0,1}-m_{0,2})(y_0(\eta))\peta y_0(\eta)|
	e^{y_2(t,\eta)}
	\int_{\eta}^{\infty}e^{-\xi}\,d\xi\,d\eta\\
	&\leq e^{2\|\zeta_2(t,\cdot)\|_C}
	\|m_{0,1}-m_{0,2}\|_{L^1}.
	\end{split}
	\end{equation*}
	Arguing similarly for 
	$\|\hat{J}_{2,1}(m_{0,1})
	-\hat{J}_{2,2}(m_{0,2})\|_{L^1}$
	and
	$\|\hat{J}_{i,1}(n_{0,1})
	-\hat{J}_{i,2}(n_{0,2})\|_{L^1}$, $i=1,2$, we obtain
	\begin{equation}\label{hatU-L1}
	\begin{split}
	\|\hat{U}_1-&\hat{U}_2\|_{L^1}
	,\|\hat{W}_1-\hat{W}_2\|_{L^1},
	\|\hat{V}_1-\hat{V}_2\|_{L^1},
	\|\hat{Z}_1-\hat{Z}_2\|_{L^1}\leq
	e^{2\max
		\{\|\zeta_1(t,\cdot)\|_C,\|\zeta_2(t,\cdot)\|_C\}}\\
	&\times\bigl(\left(\|m_{0,1}\|_{L^1}
	+\|m_{0,1}\|_{C}\|\pxi y_0\|_C
	+\|n_{0,1}\|_{L^1}
	+\|n_{0,1}\|_{C}\|\pxi y_0\|_C\right)
	\|\zeta_1(t,\cdot)-\zeta_2(t,\cdot)\|_{L^1}\bigr.\\
	&\bigl.\qquad+\|m_{0,1}-m_{0,2}\|_{L^1}
	+\|n_{0,1}-n_{0,2}\|_{L^1}\bigr).
	\end{split}
	\end{equation}
	Using \eqref{zeta-j-int} and \eqref{ab-diff}, \eqref{UZ-L1} for $U_j$, $W_j$, $V_j$, $Z_j$, $j=1,2$, we arrive at (cf.\,\,\eqref{z1-z2})
	\begin{equation*}
	\begin{split}
	\|\zeta_1-\zeta_2\|_{C([-T,T],L^1)}
	\leq& T\tilde{C}\left(\|U_1-U_2\|_{C([-T,T],L^1)}+
	\|W_1-W_2\|_{C([-T,T],L^1)}\right.\\
	&\left.\qquad+\|V_1-V_2\|_{C([-T,T],L^1)}
	+\|Z_1-Z_2\|_{C([-T,T],L^1)}\right),
	\end{split}
	\end{equation*}
	for some 
	$\tilde{C}=
	\tilde{C}(\|m_{0,j}\|_{L^1},\|n_{0,j}\|_{L^1})$.
	The latter  inequality together with \eqref{hatU-L1} and \eqref{s-zeta-a} implies \eqref{Lip-cont-zeta-a} with the norm $\|\cdot\|_{C([-T,T],L^1)}$.
	Recalling that $X^0=C\cap L^1$, we conclude that \eqref{Lip-cont-zeta-a} is proved.
	 
	Then observing that (cf.\,\,\eqref{pdpsi-zeta})
	\begin{equation}\label{pdpsi-zeta-j}
	\begin{split}
	\pxi\zeta_j(t,\xi)=\pxi y_0(\xi)-1
	-\pxi y_0(\xi)&\left(
	m_{0,j}(y_0(\xi))
	\int_0^t(\hat{Z}_j+\hat{V}_j)(\tau,\xi)\,d\tau\right.\\
	&\left.
	\quad-n_{0,j}(y_0(\xi))
	\int_0^t(\hat{U}_j-\hat{W}_j)(\tau,\xi)\,d\tau
	\right),\quad j=1,2,
	\end{split}
	\end{equation}
	and arguing as in the proof of \eqref{Lip-cont-zeta-a}, we obtain \eqref{Lip-cont-zeta-b} for $r=0$.
	Finally, successively differentiating 
	\eqref{pdpsi-zeta-j} with respect to $\xi$ and applying item (2) in Lemma \ref{U-diff} for $\hat{U}_j$, $\hat{W}_j$, $\hat{V}_j$ and $\hat{Z}_j$, we arrive at \eqref{Lip-cont-zeta-b} for $r=1,\dots,k$.
\end{proof}
\subsection{Local well-posedness}\label{lwps}

In this section we will prove that the pair $(u,v)$ defined by \eqref{uv-repr} is a unique local solution of the Cauchy problem \eqref{C-nonl-two-comp} in $X^{k+2}$.
Moreover, we will show that the data-to-solution map from the initial data $(m_0,n_0)$ to $(m,n)$ is Lipschitz continuous (see Theorem \ref{lwp} below).

We start with establishing the regularity properties of $u$ and $v$ defined by \eqref{uv-repr} as well as their decay rate for the large $|x|$.
\begin{proposition}[Regularity and decay of $u$ and $v$]
	Assume that $m_0(x),n_0(x)\in X^{k}$,
	$k\in\mathbb{N}_0$ and consider the local characteristic 
	$y(t,\xi)=\zeta(t,\xi)-\xi$ obtained in Proposition \ref{char-loc}.
	Then the functions $u(t,x)$ and $v(t,x)$, defined by \eqref{u-repr} and \eqref{v-repr} respectively, satisfy the following regularity and decay conditions (here $T>0$ is the same as in Proposition \ref{char-loc}):
	\begin{equation}\label{uv-reg}
	u,v\in
	C\left([-T,T], X^{k+2}\right)\cap
	C^1\left([-T,T], X^{k+1}\right).
	\end{equation}
\end{proposition}
\begin{proof}
	Introduce the following integrals, cf.\,\,\eqref{J-def} (recall that $\pxi y(t,\xi)\geq l$ for all 
	$t,\xi\in[-T,T]\times\mathbb{R}$ for some $l>0$, which implies that $y(t,\cdot)$ is a bijection from $\mathbb{R}$ to $\mathbb{R}$):
	\begin{equation}\label{tildeJ-def}
	\begin{split}
	&\tilde{J}_1(t,x;m_0)=
	\int_{-\infty}^{[y(t)]^{-1}(x)}
	e^{y(t,\xi)-x}m_0(y_0(\xi))
	\pxi y_0(\xi)\,d\xi,\\
	&\tilde{J}_2(t,x;m_0)=
	\int_{[y(t)]^{-1}(x)}^{\infty}
	e^{x-y(t,\xi)}m_0(y_0(\xi))
	\pxi y_0(\xi)\,d\xi,
	\end{split}
	\end{equation}
	where $\xi=[y(t)]^{-1}(x)$ is such that $y(t,\xi)=x$.
	In these notations we have (see \eqref{uv-repr})
	\begin{equation}\label{Ju}
	u(t,x)=\frac{1}{2}
	\left(\tilde{J}_1+\tilde{J}_2\right)
	(t,x;m_0),\quad
	v(t,x)=\frac{1}{2}
	\left(\tilde{J}_1+\tilde{J}_2\right)
	(t,x;n_0).
	\end{equation}
	Moreover, since
	\begin{equation}\label{tildeJ-x}
	\begin{split}
	&\px\tilde{J}_1(t,x;m_0)=
	\left.\frac{m_0(y_0(\xi))\pxi y_0(\xi)}{\pxi y(t,\xi)}\right|_{\xi=[y(t)]^{-1}(x)}
	-\tilde{J}_1(t,x;m_0),\\
	&\px\tilde{J}_2(t,x;m_0)=
	-\left.\frac{m_0(y_0(\xi))\pxi y_0(\xi)}{\pxi y(t,\xi)}\right|_{\xi=[y(t)]^{-1}(x)}
	+\tilde{J}_2(t,x;m_0),
	\end{split}
	\end{equation}
	we conclude that
	\begin{equation}\label{Ju_x}
	\px u(t,x)=\frac{1}{2}
	\left(\tilde{J}_2-\tilde{J}_1\right)
	(t,x;m_0),\quad
	\px v(t,x)=\frac{1}{2}
	\left(\tilde{J}_2-\tilde{J}_1\right)
	(t,x;n_0).
	\end{equation}
	
	Let us show that 
	$\tilde{J}_j(t,x;m_0)$ and $\tilde{J}_j(t,x;n_0)$, $j=1,2$,
	belong to the space 
	$C\left([-T,T],C^{k+1}(\mathbb{R})\right)$.
	We give a detailed proof for $\tilde{J}_1(t,x;m_0)$, the 
	other integrals can be treated similarly.
	Observing that 
	$|\tilde{J}_1(t,x;m_0)|\leq\|m_0\|_{L^1}$ and
	\begin{equation*}
	\begin{split}
	\left|\tilde{J}_1(t,x_1;m_0)
	-\tilde{J}_1(t,x_2;m_0)\right|
	&\leq\|m_0\|_C\|\pxi y_0\|_C
	\left|[y(t)]^{-1}(x_1)-[y(t)]^{-1}(x_2)\right|\\
	&\quad+\|m_0\|_{L^1}|x_1-x_2|,
	\quad x_1,x_2\in\mathbb{R},
	\end{split}
	\end{equation*}
	where we have used \eqref{exp-ineq},
	we conclude that 
	$\tilde{J}_1(t,x;m_0)\in 
	L^\infty([-T,T],C(\mathbb{R}))$.
	
	Then take any $t_1,t_2\in[-T,T]$ and let 
	$\xi_j=\xi_j(x)=[y(t_j)]^{-1}(x)$, $j=1,2$.
	Since $y(t_1,\xi_1)=y(t_2,\xi_2)=x$ and using the 
	mean value theorem, we obtain
	\begin{equation}\label{xi-diff}
	\xi_1-\xi_2=\frac
	{y(t_1,\xi_1)-y(t_2,\xi_2)}
	{\pxi y(t_1,\xi^*)}
	=\frac
	{\pt y(t^*,\xi_2)(t_2-t_1)}
	{\pxi y(t_1,\xi^*)},
	\end{equation}
	for certain values of $t^*$ and $\xi^*$ lying between 
	$t_1$ and $t_2$, and $\xi_1$ and $\xi_2$, respectively.
	We have the following inequality (here the arguments 
	$x,m_0$ of $\tilde{J}_1(t,x;m_0)$ are dropped):
	\begin{equation}\label{J-1-diff}
	\begin{split}
	\left|\tilde{J}_1(t_1)-\tilde{J}_1(t_2)\right|
	&\leq\left|
	\int_{\xi_1}^{\xi_2}
	e^{y(t_1,\xi)-x}m_0(y_0(\xi))
	\pxi y_0(\xi)\,d\xi
	\right|\\
	&\quad+
	\int_{-\infty}^{\xi_2}
	\left|e^{y(t_1,\xi)-x}-e^{y(t_2,\xi)-x}\right|
	\left|m_0(y_0(\xi))
	\pxi y_0(\xi)\right|\,d\xi=I_5+I_6.
	\end{split}
	\end{equation}
	
	Recalling that $x=y(t_1,\xi_1)$ and $\pxi y(t,\xi)\geq l$ and applying 
	the mean value theorem, the integral $I_5$ can be estimated 
	as follows (here we denote
	$\tilde{M}_0=\|m_0\|_C\|\pxi y_0\|_C$):
	\begin{equation}\label{I-3-est}
	\begin{split}
	I_5&\leq\tilde{M}_0e^{-x}\left|
	\int_{\xi_1}^{\xi_2}
	e^{y(t_1,\xi)}\,d\xi\right|
	\leq \frac{\tilde{M}_0}{l}e^{-x}
	\left|\int_{y(t_1,\xi_1)}^{y(t_1,\xi_2)}
	e^z\,dz\right|\\
	&\leq\frac{\tilde{M}_0}{l}
	e^{-y(t_1,\xi_1)+y(t_1,\xi^*)}
	\left|\pxi y(t_1,\xi^*)\right||\xi_1-\xi_2|
	\leq\frac{\tilde{M}_0}{l}
	\|\partial_{(\cdot)} y(t_1,\cdot)\|_C
	e^{2\|\zeta(t_1,\cdot)\|_C+\xi^*-\xi_1}
	|\xi_1-\xi_2|.
	\end{split}
	\end{equation}
	Taking into account that 
	$e^{\xi^*-\xi_1}\leq e^{|\xi_1-\xi_2|}$, we have from \eqref{xi-diff} 
	and \eqref{I-3-est} that
	\begin{equation}\label{I-3-est-1}
	I_5\leq \frac{\tilde{M}_0}{l^2}
	\|\partial_{(\cdot)} y(t_1,\cdot)\|_C
	\|\pt y(t^*,\cdot)\|_C
	\exp\left\{2\|\zeta(t_1,\cdot)\|_C+
	\frac{\|\pt y(t^*,\cdot)\|_C}
	{l}|t_1-t_2|\right\}
	|t_1-t_2|.
	\end{equation}
	
	The integral $I_6$ can be estimated as follows
	(as above, we denote
	$\tilde{M}_0=\|m_0\|_C\|\pxi y_0\|_C$):
	\begin{equation}\label{I-4-est}
	\begin{split}
	I_6&\leq\tilde{M}_0e^{-x}
	\int_{-\infty}^{\xi_2(x)}
	\left|e^{y(t_1,\xi)}-e^{y(t_2,\xi)}\right|\,d\xi
	\leq\tilde{M}_0e^{-x}
	\|\pt\zeta(t^*,\cdot)\|_C
	e^{\|\zeta(t^*,\cdot)\|_C}
	\int_{-\infty}^{\xi_2(x)}e^\xi\,d\xi
	|t_1-t_2|,\\
	&\leq \|\pt\zeta(t^*,\cdot)\|_C
	e^{\|\zeta(t^*,\cdot)\|_C+\|\zeta(t_2,\cdot)\|_C}
	|t_1-t_2|,
	\end{split}
	\end{equation}
	for some $t^*$ between $t_1$ and $t_2$
	(here we have used that 
	$|[y(t_2)]^{-1}(x)-x|\leq\|\zeta(t_2,\cdot)\|_C$).
	
	Combining \eqref{J-1-diff}, \eqref{I-3-est-1} and \eqref{I-4-est} we conclude that 
	$\tilde{J}_1(t,x;m_0)\in C([-T,T],C(\mathbb{R}))$.
	Then \eqref{tildeJ-x} implies that 
	$\tilde{J}_1(t,x;m_0)\in 
	C\left([-T,T],C^1(\mathbb{R})\right)$.
	Using \eqref{zeta-reg} and successively differentiating 
	\eqref{tildeJ-x} with respect to $x$, we conclude that
	$\tilde{J}_1(t,x;m_0)\in 
	C\left([-T,T],C^{k+1}(\mathbb{R})\right)$.
	Arguing similarly, we can prove that
	$\tilde{J}_2(m_0),\tilde{J}_j(n_0)\in 
	C\left([-T,T],C^{k+1}\right)$,
	$j=1,2$, which, together with \eqref{Ju} and \eqref{Ju_x}, imply that
	$u,v\in C\left([-T,T], C^{k+2}\right)$.
	
	Now let us prove that 
	$\tilde{J}_j(m_0),\tilde{J}_j(n_0)$, $j=1,2$,
	belong to the space 
	$C\left([-T,T],
	W^{k+1,1}(\mathbb{R})\right)$.
	As above, we give a detailed proof for $\tilde{J}_1(m_0)$, all 
	other integrals can be analyzed in a similar manner.
	By Fubini's theorem,
	\begin{equation*}
	\begin{split}
	\int_{-\infty}^{\infty}\left|
	\tilde{J}_1(t,x;m_0)\right|\,dx
	&\leq
	\int_{-\infty}^{\infty}
	\int_{-\infty}^{[y(t)]^{-1}(x)}\left|
	e^{y(t,\xi)-x}m_0(y_0(\xi))\pxi y_0(\xi)
	\right|\,d\xi\,dx\\
	&=\int_{-\infty}^{\infty}
	e^{y(t,\xi)}
	|m_0(y_0(\xi))\pxi y_0(\xi)|
	\int^{\infty}_{y(t,\xi)}e^{-x}
	\,dx\,d\xi=\|m_0\|_{L^1},
	\end{split}
	\end{equation*}
	which implies that 
	$\tilde{J}_1\in
	L^\infty\left([-T,T],L^1\right)$.
	
	For any $t_1,t_2\in[-T,T]$ we have (recall the notation 
	$\xi_j=[y(t_j)]^{-1}(x)$, $j=1,2$; here we drop the arguments $x,m_0$ of 
	$\tilde{J}_1(t,x;m_0)$ for simplicity):
	\begin{equation}\label{J-diff-L1}
	\begin{split}
	\|\tilde{J}_1(t_1)-\tilde{J}_1(t_2)\|_{L^1}
	\leq&\int_{-\infty}^{\infty}
	\left|\int_{\xi_1}^{\xi_2}e^{y(t_1,\xi)-x}
	|m_0(y_0(\xi))\pxi y_0(\xi)|\,d\xi
	\right|\,dx\\
	&+\int_{-\infty}^{\infty}\int_{-\infty}^{\xi_2}
	\left|e^{y(t_1,\xi)-x}-e^{y(t_2,\xi)-x}\right|
	|m_0(y_0(\xi))\pxi y_0(\xi)|\,d\xi
	\,dx=I_7+I_8.
	\end{split}
	\end{equation}
	Applying the Fubini's theorem and the mean value theorem, we obtain
	\begin{equation}\label{I-5-est}
	\begin{split}
	I_7&=\int_{-\infty}^{\infty}
	e^{y(t_1,\xi)}|m_0(y_0(\xi))\pxi y_0(\xi)|
	\left|\int_{y(t_1,\xi)}^{y(t_2,\xi)}e^{-x}\,dx
	\right|
	\,d\xi\\
	&=
	\int_{-\infty}^{\infty}
	e^{\zeta(t_1,\xi)}
	\left|e^{-\zeta(t_1,\xi)}
	-e^{-\zeta(t_2,\xi)}\right|
	|m_0(y_0(\xi))\pxi y_0(\xi)|\,d\xi\\
	&\leq e^{\|\zeta(t_1,\cdot)\|_C
		+\|\zeta(t^*,\cdot)\|_C}
	\|\pt\zeta(t^*,\cdot)\|_C\|m_0\|_{L^1}|t_1-t_2|,
	\end{split}
	\end{equation}
	for some $t^*$ between $t_1$ and $t_2$.
	Using similar arguments for $I_8$, we arrive at the following inequality:
	\begin{equation}\label{I-6-est}
	\begin{split}
	I_8&=\int_{-\infty}^{\infty}
	\left|e^{y(t_1,\xi)}-e^{y(t_2,\xi)}\right|
	|m_0(y_0(\xi))\pxi y_0(\xi)|
	\int_{y(t_2,\xi)}^{\infty}e^{-x}\,dx
	\,d\xi\\
	&=\int_{-\infty}^{\infty}e^{\zeta(t_2,\xi)}
	\left|e^{\zeta(t_1,\xi)}-e^{\zeta(t_2,\xi)}\right|
		|m_0(y_0(\xi))\pxi y_0(\xi)|\,d\xi\\
	&\leq e^{\|\zeta(t_1,\cdot)\|_C
		+\|\zeta(t^*,\cdot)\|_C}
	\|\pt\zeta(t^*,\cdot)\|_C\|m_0\|_{L^1}|t_1-t_2|,
	\end{split}
	\end{equation}
	for some $t^*$ between $t_1$ and $t_2$.
	Combining \eqref{J-diff-L1}, \eqref{I-5-est} and
	\eqref{I-6-est}, we have that
	$\tilde{J}_1(m_0)\in
	C\left([-T,T],L^1\right)$. 
	In view of \eqref{tildeJ-x}, we eventually conclude that 
	$\tilde{J}_1(m_0)\in
	C\left([-T,T],W^{k+1,1}\right)$.
	Arguing similarly for $\tilde{J}_2(m_0)$
	and $\tilde{J}_j(n_0)$, $j=1,2$, the equations \eqref{Ju} and \eqref{Ju_x} 
	imply that $u,v\in C\left([-T,T],W^{k+2,1}\right)$.
	
	Finally, let us show that 
	$\pt u,\pt v\in C([-T,T],X^{k+1})$.
	Introducing
	\begin{equation*}
	\begin{split}
	&\check{J}_1(t,x;m_0)=
	\int_{-\infty}^{[y(t)]^{-1}(x)}
	e^{y(t,\xi)-x}
	\pt y(t,\xi)m_0(y_0(\xi))
	\pxi y_0(\xi)\,d\xi,\\
	&\check{J}_2(t,x;m_0)=
	\int_{[y(t)]^{-1}(x)}^{\infty}
	e^{x-y(t,\xi)}
	\pt y(t,\xi)m_0(y_0(\xi))
	\pxi y_0(\xi)\,d\xi,
	\end{split}
	\end{equation*}
	we obtain that (cf.\,\,\eqref{Ju_x})
	\begin{equation}\label{checkJu}
	\pt u(t,x)=\frac{1}{2}
	\left(\check{J}_1-\check{J}_2\right)
	(t,x;m_0),\quad
	\pt v(t,x)=\frac{1}{2}
	\left(\check{J}_1-\check{J}_2\right)
	(t,x;n_0).
	\end{equation}
	Condition \eqref{zeta-reg} implies that
	$\pt y(t,\xi)m_0(y_0(\xi))
	\pxi y_0(\xi)\in
	C\left([-T,T],X^{k}\right)$.
	Therefore using
	\begin{equation*}
	\begin{split}
	&\px\check{J}_1(t,x;m_0)=
	\left.\frac{\pt y(t,\xi)m_0(\xi)\pxi y_0(\xi)}
	{\pxi y(t,\xi)}\right|_{\xi=[y(t)]^{-1}(x)}
	-\check{J}_1(t,x;m_0),\\
	&\px\check{J}_2(t,x;m_0)=
	-\left.\frac{\pt y(t,\xi)m_0(\xi)\pxi y_0(\xi)}
	{\pxi y(t,\xi)}\right|_{\xi=[y(t)]^{-1}(x)}
	+\check{J}_2(t,x;m_0),
	\end{split}
	\end{equation*}
	and applying a similar line of reasoning as we did for 
	$\tilde{J}_1(m_0)$ earlier, we conclude that 
	$\check{J}_j(m_0),\check{J}_j(n_0)$ belong to 
	$C\left([-T,T],X^{k+1}\right)$.
	which, together with \eqref{checkJu}, implies that
	$\pt u,\pt v\in
	C\left([-T,T],X^{k+1}\right)$.
\end{proof}


We are now in a position to demonstrate that $u$ and $v$, as defined 
in \eqref{uv-repr} through the characteristics $y$, constitute a solution 
to the Cauchy problem \eqref{C-nonl-two-comp}.

\begin{proposition}[Local existence]\label{loc-sol}
	Suppose that $m_0(x),n_0(x)\in X^{k}$,
	$k\in\mathbb{N}_0$. Consider the local characteristics 
	$y(t,\xi)=\zeta(t,\xi)-\xi$ obtained in Proposition \ref{char-loc} and
	define the functions $u(t,x)$ and $v(t,x)$ by \eqref{u-repr} and \eqref{v-repr} 
	respectively. Then $u$ and $v$ satisfy \eqref{uv-reg} and the pair $(u,v)$ is 
	a solution of the Cauchy problem 
	 \eqref{C-nonl-two-comp}.
\end{proposition}
\begin{proof}
	Given that $y(0,x)=y_0(\xi)$, it becomes evident that 
	the functions $u$ and $v$ defined by \eqref{uv-repr} fulfill 
	the initial conditions $u(0,x)=u_0(x)$ and $v(0,x)=v_0(x)$. 
	Subsequently, we establish \eqref{nonl-two-comp} for $u$, 
	with analogous reasoning being applicable to $v$.
	Using \eqref{cons-u}, which follows from \eqref{UV}, we have
	\begin{equation}\label{m-repr}
	m(t,x)=\int_{-\infty}^{\infty}
	\delta(x-z)m(t,x)\,dz=
	\int_{-\infty}^{\infty}
	\delta(x-y(t,\xi))
	m_0(y_0(\xi))\pxi y_0(\xi)\,d\xi,
	\end{equation}
	where $\delta(x)$ is the Dirac delta function.
	Utilizing \eqref{m-repr}, we obtain for any 
	$\phi(x)\in \mathcal{S}(\mathbb{R})$,
	\begin{equation*}
	\begin{split}
	(\pt u,\phi)&
	=\pt\left((1-\px^2)^{-1}m(t,x),\phi(x)\right)=
	\pt\left(m(t,x),(1-\px^2)^{-1}\phi(x)\right)\\
	&=\pt\int_{-\infty}^{\infty}
	m_0(y_0(\xi))\pxi y_0(\xi)
	\int_{-\infty}^{\infty}
	\delta(x-y(t,\xi))(1-\px^2)^{-1}\phi(x)
	\,dx\,d\xi\\
	&=\int_{-\infty}^{\infty}
	(W-U)(Z+V)(t,\xi)m(t,y(t,\xi))\pxi y(t,\xi)\cdot
	(1-\px^2)^{-1}\px\phi(y(t,\xi))\,d\xi\\
	&=\left(
	(1-\px^2)^{-1}\px
	[m(u-\partial_xu)(v+\partial_xv)],\phi
	\right),
	\end{split}
	\end{equation*}
	which implies \eqref{nonl-two-comp} for $u$.
\end{proof}

To establish the uniqueness of the weak solution for the FORQ 
equation in $W^{2,1}(\R)$, one can typically rely on its representation 
as a first-order equation, as discussed in \cite[Section 4]{Zh16} 
and \cite[Section 4.1]{GL17}. However, the two-component 
system \eqref{two-comp} (and the nonlocal FORQ equation) 
cannot be converted into a first-order equation, 
as is evident from terms like $(u\partial_x^2z-(\partial_x^2w)v)$ in \eqref{nonl-terms}.
Therefore, in order to establish the uniqueness of \eqref{C-nonl-two-comp}, 
alternative arguments need to be employed. Specifically, by adhering to 
the Lagrangian approach, we demonstrate that any solution 
of \eqref{C-nonl-two-comp} within the class \eqref{uv-reg}, 
must take the form \eqref{uv-repr}.

\begin{lemma}\label{cl}
	Suppose that $(u,v)$ is a local solution of the Cauchy 
	problem \eqref{C-nonl-two-comp}, which satisfy \eqref{uv-reg} 
	for some $T>0$ and $k\in\mathbb{N}_0$.
	Consider $(y_0(\xi)-\xi)\in E_c$ for some $c>0$.
	Then 
	\begin{enumerate}
		\item there exists a unique solution 
		$y(t,\xi)$ of \eqref{y-iv} subject to the initial data $y_0(\xi)$, such that
		$(y(t,\xi)-\xi)\in C^1([-T,T]\times\mathbb{R})$
		and $\pxi y(t,\xi)>l$ for all
		$t,\xi\in[-T,T]\times\mathbb{R}$, where $0<l<c$, and $T>0$ 
		is sufficiently small;
		
		\item the equalities \eqref{cons-uv} hold.
	\end{enumerate}
\end{lemma}
\begin{proof}
	The vector field $(\px u-u)(\px v+v)$ in \eqref{y-iv} is bounded 
	in $x$ and is of class $C^1$ in $x$ and $t$.
	Therefore the classical Cauchy theorem for ODEs implies 
	that there exists a unique solution $y(t,\xi)$,
	$(y(t,\xi)-\xi)\in C^1([-T,T]\times\mathbb{R})$, of the Cauchy 
	problem for \eqref{y-iv} with the initial data $y_0(\xi)$.
	Since $\pxi y_0(\xi)\geq c$ for all $\xi\in\mathbb{R}$ 
	and $\pxi y(t,\xi)\in C([-T,T]\times\mathbb{R})$, we have item (1) of the lemma.
	
	Let us demonstrate item (2). Consider 
	$k\in\N$. We focus on the classical solution to the Cauchy 
	problem for \eqref{two-comp}. Given this, we can directly establish 
	the validity of \eqref{cons-0} and, thus, \eqref{cons-uv}.

	In the case $k=0$ we have that 
	$m(t,\cdot)$, $n(t,\cdot)\in X^{0}$ and 
	$\pt\px u(t,\cdot)$, $\pt\px v(t,\cdot)\in X^{0}$ for all fixed $t\in[-T,T]$.
	Therefore \eqref{two-comp} can be considered for 
	all fixed $t$ as an equality of functionals acting 
	on $W^{1,\infty}(\mathbb{R})$, i.e.,
	\begin{equation}\label{distr-two-comp}
	\begin{split}
	&(\partial_tm(t,x),\phi(x))=
	\left(\partial_x[m(u-\partial_xu)
	(v+\partial_xv)](t,x),\phi(x)\right),\\
	&(\partial_tn(t,x),\phi(x))=
	\left(\partial_x[n(u-\partial_xu)
	(v+\partial_xv)](t,x),\phi(x)\right),
	\end{split}
	\end{equation}
	for any $\phi\in W^{1,\infty}(\mathbb{R})$. 
	To enhance clarity in distinguishing between the various variables 
	in the functionals above, we have opted to explicitly write 
	the variable $x$, even though it would be more accurate to omit 
	$x$ or replace $x$ by ``$\cdot$".
	Moreover, since $\pxi y(t,\xi)>l$ for all 
	$t,\xi\in[-T,T]\times\mathbb{R}$, we conclude that
	$m(t,y(t,\cdot))$, $n(t,y(t,\cdot))\in 
	L^1(\mathbb{R})$ and 
	$\pt\px u(t,y(t,\cdot))$,
	$\pt\px v(t,y(t,\cdot))\in
	L^1(\mathbb{R})$. 
	Therefore we have for any $\phi\in W^{1,\infty}$
	\begin{equation}\label{m-x-y}
	\begin{split}
	&\left(
	\frac{d}{d\xi}m(t,y(t,\xi)),\phi(\xi)
	\right)=
	\left(
	\px m(t,y(t,\xi))\pxi y(t,\xi),\phi(\xi)
	\right),\\
	&\left(
	\frac{d}{d\xi}n(t,y(t,\xi)),\phi(\xi)
	\right)=
	\left(
	\px n(t,y(t,\xi))\pxi y(t,\xi),\phi(\xi)
	\right),
	\end{split}
	\end{equation}
	and
	\begin{equation}\label{m-t-y}
	\begin{split}
	&\left(
	\frac{d}{d\xi}\pt\px u(t,y(t,\xi)),\phi(\xi)
	\right)=
	\left(
	\pt m(t,y(t,\xi))\pxi y(t,\xi),\phi(\xi)
	\right),\\
	&\left(
	\frac{d}{d\xi}\pt\px v(t,y(t,\xi)),\phi(\xi)
	\right)=
	\left(
	\pt n(t,y(t,\xi))\pxi y(t,\xi),\phi(\xi)
	\right).
	\end{split}
	\end{equation}
	Thus the right hand side of \eqref{m-x-y} and \eqref{m-t-y} exist, even 
	though $\pxi y(t,\xi)$ is just a continuous function of $\xi$.
	Taking into account that $\phi([y(t)]^{-1}(x))\in W^{1,\infty}$, 
	as soon as $\phi\in W^{1,\infty}$, we can take $\phi([y(t)]^{-1}(x))$ 
	instead of $\phi(x)$ in \eqref{distr-two-comp}.
	Changing the variables $x=y(t,\xi)$, we obtain
	\begin{equation*}
	\begin{split}
	&(\partial_tm(t,y(t,\xi))
	\pxi y(t,\xi),\phi(\xi))=
	\left(\partial_x[m(u-\partial_xu)
	(v+\partial_xv)](t,y(t,\xi))
	\pxi y(t,\xi),\phi(\xi)\right),\\
	&(\partial_tn(t,y(t,\xi))
	\pxi y(t,\xi),\phi(\xi))=
	\left(\partial_x[n(u-\partial_xu)
	(v+\partial_xv)](t,y(t,\xi))
	\pxi y(t,\xi),\phi(\xi)\right),
	\end{split}
	\end{equation*}
	which immediately implies \eqref{cons-0} and thus \eqref{cons-uv}.
\end{proof}
Using Lemma \ref{cl}, we can prove uniqueness of the solution of \eqref{C-nonl-two-comp}.
\begin{proposition}[Uniqueness]\label{Uni}
	The local solution $(u,v)$ of the Cauchy problem \eqref{C-nonl-two-comp} is unique in the class \eqref{uv-reg} for any $k\in\mathbb{N}_0$.
\end{proposition}
\begin{proof}
	Lemma \ref{cl} implies that such a solution $(u,v)$ has the 
	representation \eqref{uv-repr}. Therefore \eqref{zeta-iv} for $\zeta=y-\xi$ 
	is equivalent to the Cauchy problem for \eqref{y-iv} with 
	initial data $y(0,\xi)=y_0(\xi)$ (here $y$ is the same as in Lemma \ref{cl}).
	Since the vector field in \eqref{zeta-iv} depends on the initial data $m_0,n_0$ 
	only and, according to Proposition \ref{char-loc}, the Cauchy 
	problem \eqref{zeta-iv} has a unique solution, we conclude that the 
	characteristic obtained in Lemma \ref{cl} is the same as that 
	obtained in Proposition \ref{char-loc}.
	This implies that any solution $(u,v)$ in the considered 
	class is that obtained in Proposition \ref{loc-sol}.
\end{proof}

\begin{proposition}
	\label{Lip-cont}
	Fix any two constants $0<R_0< R$.
	Suppose that $m_{0,j}(x),n_{0,j}(x)\in X^{k}$, $j=1,2$, 
	for some $k\in\mathbb{N}_0$, are such that
	$$
	\|m_{0,j}\|_{X^{0}},
	\|n_{0,j}\|_{X^{0}}\leq R_0,\quad
	\mbox{and}\quad
	\|m_{0,j}\|_{X^{k}},
	\|n_{0,j}\|_{X^{k}}\leq R,\quad j=1,2.
	$$
	Consider the two corresponding local solutions $(u_j,v_j)(t,x)$ 
	of \eqref{C-nonl-two-comp} in the class \eqref{uv-reg}.
	Then the data-to-solution map satisfy the following conditions:
	\begin{subequations}\label{Lip-cont-sol}
	\begin{align}
	\label{Lip-cont-sol-a}
	&\|u_1-u_2\|
	_{C([-T,T],X^1)},
	\|v_1-v_2\|_{C([-T,T],X^1)}
	\leq C_0\left(\|m_{0,1}-m_{0,2}\|_{L^{1}}
	+\|n_{0,1}-n_{0,2}\|_{L^{1}}
	\right),\\
	&\label{Lip-cont-sol-b}
	\|u_1-u_2\|
	_{C([-T,T],X^{k+2})},
	\|v_1-v_2\|
	_{C([-T,T],X^{k+2})}
	\leq C\left(\|m_{0,1}-m_{0,2}\|
	_{X^k}
	+\|n_{0,1}-n_{0,2}\|
	_{X^{k}}
	\right),
	\end{align}
	\end{subequations}
	for some $C_0=C_0(T,R_0)>0$, $C=C(T,R)>0$ and 
	sufficiently small $T>0$.
\end{proposition}
\begin{corollary}[Lipschitz continuity]
	\label{Lip-c}
	It is easy to see that \eqref{Lip-cont-sol} imply
	the following Lipschitz property for the solutions 
	$(m_j,n_j)(t,x)=(1-\px^2)(u_j,v_j)(t,x)$, $j=1,2$:
	\begin{equation*}
	\|m_1-m_2\|
	_{C([-T,T],X^{k})},
	\|n_1-n_2\|
	_{C([-T,T],X^{k})}
	\leq C\left(\|m_{0,1}-m_{0,2}\|_{X^{k}}
	+\|n_{0,1}-n_{0,2}\|_{X^{k}}
	\right).
	\end{equation*}
\end{corollary}
\begin{proof}
	Consider the two characteristics $\zeta_j(t,\xi)=y_j(t,\xi)-\xi$ obtained 
	in Proposition \ref{char-loc}, which correspond to $(m_{0,j}, n_{0,j})$, $j=1,2$ 
	(see \eqref{zeta-j-int} in Proposition \ref{Lip-cont-ch}).
	According to Propositions \ref{loc-sol} and \ref{Uni}, the solutions 
	$u_j,v_j$ have the representation \eqref{uv-repr}.
	Therefore, by introducing the integrals 
	(cf.\,\,\eqref{tildeJ-def})
	\begin{equation*}
	\begin{split}
	&\breve{J}_{1,j}(t,x;m_{0,j})=
	\int_{-\infty}^{\xi_j}
	e^{y_j(t,\xi)-x}m_{0,j}(y_0(\xi))
	\pxi y_0(\xi)\,d\xi,\quad j=1,2,\\
	&\breve{J}_{2,j}(t,x;m_{0,j})=
	\int_{\xi_j}^{\infty}
	e^{x-y_j(t,\xi)}m_{0,j}(y_0(\xi))
	\pxi y_0(\xi)\,d\xi,\quad j=1,2,
	\end{split}
	\end{equation*}
	where 
	\begin{equation}\label{xi-j}
	\xi_j=[y_j(t)]^{-1}(x),\quad j=1,2,
	\end{equation}
	we have (cf.\,\,\eqref{Ju})
	\begin{equation}\label{Ju-j}
	u_j(t,x)=\frac{1}{2}
	\left(\breve{J}_{1,j}+\breve{J}_{2,j}\right)
	(t,x;m_{0,j}),\quad
	v_j(t,x)=\frac{1}{2}
	\left(\breve{J}_{1,j}+\breve{J}_{2,j}\right)
	(t,x;n_{0,j}),\quad j=1,2.
	\end{equation}
	Using (cf.\,\,\eqref{tildeJ-x})
	\begin{equation}\label{breveJ-x}
	\begin{split}
	&\px\breve{J}_{1,j}(t,x;m_{0,j})=
	\frac{m_{0,j}(y_0(\xi_j))\pxi y_0(\xi_j)}
	{\pxi y_j(t,\xi_j)}
	-\breve{J}_{1,j}(t,x;m_{0,j}),\quad j=1,2,\\
	&\px\breve{J}_{2,j}(t,x;m_{0,j})=
	-\frac{m_{0,j}(y_0(\xi_j))\pxi y_0(\xi_j)}
	{\pxi y_j(t,\xi_j)}
	+\breve{J}_{2,j}(t,x;m_{0,j}),\quad j=1,2,
	\end{split}
	\end{equation}
	we conclude that (cf.\,\,\eqref{Ju_x})
	\begin{equation}\label{Ju-j_x}
	\px u_j(t,x)=\frac{1}{2}
	\left(\breve{J}_{2,j}-\breve{J}_{1,j}\right)
	(t,x;m_{0,j}),\quad
	\px v_j(t,x)=\frac{1}{2}
	\left(\breve{J}_{2,j}-\breve{J}_{1,j}\right)
	(t,x;n_{0,j}),\quad j=1,2.
	\end{equation}
	
	First, let us show \eqref{Lip-cont-sol-a} for
	$C\left([-T,T],C^1\right)$.
	Observe that (we drop the arguments $t,x$ of $\breve{J}_{1,j}(t,x;m_{0,j})$ for simplicity)
	\begin{equation}\label{breveJ-diff-1}
	\begin{split}
	&\left|\breve{J}_{1,1}(m_{0,1})-\breve{J}_{1,2}(m_{0,2})\right|
	\leq\left|\left(
	\int_{-\infty}^{\xi_1}-\int_{-\infty}^{\xi_2}\right)
	e^{y_1(t,\xi)-x}m_{0,1}(y_0(\xi))\pxi y_0(\xi)\,d\xi
	\right|\\
	&\qquad\qquad+\left|\int_{-\infty}^{\xi_2}
	\left(e^{y_1(t,\xi)-x}m_{0,1}(y_0(\xi))-
	e^{y_2(t,\xi)-x}m_{0,2}(y_0(\xi))\right)\pxi y_0(\xi)
	\,d\xi\right|
	=I_9+I_{10}.
	\end{split}
	\end{equation}
	Taking into account that (recall the definition of $\xi_j$ given in \eqref{xi-j})
	\begin{align*}
	&l|\xi_1-\xi_2|\leq |y_1(t,\xi_1)-y_2(t,\xi_1)|
	=|\zeta_1(t,\xi_1)-\zeta_2(t,\xi_1)|,
	\end{align*}
	we have for $I_9$:
	\begin{equation}\label{breve-1}
	\begin{split}
	I_9&\leq\|m_{0,1}\|_C\|\pxi y_0\|_Ce^{-x}
	e^{\|\zeta_1(t,\cdot)\|_C}
	\left|
	\int_{\xi_1}^{\xi_2}e^{\xi}
	\right|\leq
	\|m_{0,1}\|_C\|\pxi y_0\|_C
	e^{\max\{\xi_1,\xi_2\}-x}
	|\xi_1-\xi_2|\\
	&\leq \frac{\|m_{0,1}\|_C}{l}\|\pxi y_0\|_C
	e^{\max\{\|\zeta_1(t,\cdot)\|_C,
		\|\zeta_2(t,\cdot)\|_C\}}
	\|\zeta_1(t,\cdot)-\zeta_2(t,\cdot)\|_C,
	\end{split}
	\end{equation}
	where we have used that 
	$\left|[y_j(t)]^{-1}(x)-x\right|
	\leq\|\zeta_j(t,\cdot)\|_C$, $j=1,2$.
	The integral $I_{10}$ can be estimated as follows:
	\begin{equation}\label{breve-2}
	\begin{split}
	I_{10}&\leq\left|\int_{-\infty}^{\xi_2}
	\left(e^{y_1(t,\xi)-x}-e^{y_2(t,\xi)-x}\right)
	m_{0,1}(y_0(\xi))\pxi y_0(\xi)\,d\xi\right|\\
	&\quad+\left|\int_{-\infty}^{\xi_2}
	e^{y_2(t,\xi)-x}(m_{0,1}-m_{0,2})(y_0(\xi))\pxi y_0(\xi)
	\,d\xi\right|=I_{10,1}+I_{10,2}.
	\end{split}
	\end{equation}
	Here
	\begin{equation*}
	I_{10,2}\leq\|\pxi y_0\|_C\|m_{0,1}-m_{0,2}\|_{L^1},
	\end{equation*}
	and (recall \eqref{exp-ineq-b} and that $|\xi_2-x|=\left|[y_2(t)]^{-1}(x)-x\right|
	\leq\|\zeta_2(t,\cdot)\|_C$)
	\begin{equation*}
	\begin{split}
	I_{10,1}&\leq \int_{-\infty}^{\xi_2}e^{\xi-x}
	\left|e^{\zeta_1(t,\xi)}-e^{\zeta_2(t,\xi)}\right|
	|m_{0,1}(y_0(\xi))\pxi y_0(\xi)|\,d\xi\\
	&\leq e^{\xi_2-x}\|m_{0,1}\|_{L^1}
	e^{\max\{\|\zeta_1(t,\cdot)\|_C,\|\zeta_2(t,\cdot)\|_C\}}
	\|\zeta_1(t,\cdot)-\zeta_2(t,\cdot)\|_{C}\\
	&\leq \|m_{0,1}\|_{L^1}
	e^{2\max\{\|\zeta_1(t,\cdot)\|_C,\|\zeta_2(t,\cdot)\|_C\}}
	\|\zeta_1(t,\cdot)-\zeta_2(t,\cdot)\|_{C}.
	\end{split}
	\end{equation*}
	Combining \eqref{breveJ-diff-1}, \eqref{breve-1}, \eqref{breve-2}, \eqref{s-zeta-a} and \eqref{Lip-cont-zeta-a}, we obtain
	\begin{equation*}
	\left|\breve{J}_{1,1}(m_{0,1})-\breve{J}_{1,2}(m_{0,2})\right|
	\leq C_0\left(\|m_{0,1}-m_{0,2}\|_{L^1}
	+\|n_{0,1}-n_{0,2}\|_{L^1}\right),
	\quad C_0=C_0(T,R_0)>0,
	\end{equation*}
	with $T>0$ sufficiently small.
	Arguing similarly for 
	$\breve{J}_{2,1}(m_{0,1})-\breve{J}_{2,2}(m_{0,2})$ and
	$\breve{J}_{i,1}(n_{0,1})-\breve{J}_{i,2}(n_{0,2})$, $i=1,2$, we eventually arrive at 
	\begin{equation}\label{breveJ-diff}
	\begin{split}
	&\left|\breve{J}_{i,1}(t,x;m_{0,1})
	-\breve{J}_{i,2}(t,x;m_{0,2})\right|
	\leq C_0(\|m_{0,1}-m_{0,2}\|_{L^1}+\|n_{0,1}-n_{0,2}\|_{L^1}),
	\quad i=1,2,\\
	&\left|\breve{J}_{i,1}(t,x;n_{0,1})
	-\breve{J}_{i,2}(t,x;n_{0,2})\right|
	\leq C_0(\|m_{0,1}-m_{0,2}\|_{L^1}+\|n_{0,1}-n_{0,2}\|_{L^1})
	,\quad\,\,\, i=1,2,
	\end{split}
	\end{equation}
	for all $t,x\in[-T,T]\times\mathbb{R}$ and for some $C_0=C_0(T,R_0)>0$.
	Combining 
	\eqref{Ju-j}, \eqref{Ju-j_x} and \eqref{breveJ-diff}, we obtain \eqref{Lip-cont-sol-a} for 
	$C\left([-T,T],C^1\right)$.
	
	To prove \eqref{Lip-cont-sol-a} for
	$C\left([-T,T],W^{1,1}\right)$, 
	we observe that (cf.\,\,\eqref{breveJ-diff-1})
	\begin{equation}\label{breveJ-diff-2}
	\begin{split}
	&\left\|\breve{J}_{1,1}(m_{0,1})-\breve{J}_{1,2}(m_{0,2})\right\|_{L^1}
	\leq\int_{\infty}^{\infty}\left|\left(
	\int_{-\infty}^{\xi_1}-\int_{-\infty}^{\xi_2}\right)
	e^{y_1(t,\xi)-x}m_{0,1}(y_0(\xi))\pxi y_0(\xi)\,d\xi
	\right|dx\\
	&\qquad+\int_{\infty}^{\infty}
	\left|\int_{-\infty}^{\xi_2}
	\left(e^{y_1(t,\xi)-x}m_{0,1}(y_0(\xi))-
	e^{y_2(t,\xi)-x}m_{0,2}(y_0(\xi))\right)\pxi y_0(\xi)
	\,d\xi\right|dx
	=I_{11}+I_{12}.
	\end{split}
	\end{equation}
	Changing the order of integration, we have the following estimate for $I_{11}$:
	\begin{equation}\label{breve-3}
	\begin{split}
	I_{11}&\leq\int_{-\infty}^{\infty}
	e^{y_1(t,\xi)}|m_{0,1}(y_0(\xi))\pxi y_0(\xi)|
	\left|
	\int_{y_1(t,\xi)}^{y_2(t,\xi)}e^{-x}\,dx
	\right|d\xi\\
	&=\int_{-\infty}^{\infty}
	e^{\zeta_1(t,\xi)}
	\left|
	e^{-\zeta_1(t,\xi)}-e^{-\zeta_2(t,\xi)}
	\right|
	|m_{0,1}(y_0(\xi))\pxi y_0(\xi)|d\xi\\
	&\leq \|m_{0,1}\|_{L^1}
	e^{2\max\{\|\zeta_1(t,\cdot)\|_C,
		\|\zeta_2(t,\cdot)\|_C\}}
	\|\zeta_1(t,\cdot)-\zeta_2(t,\cdot)\|_C.
	\end{split}
	\end{equation}
	As in \eqref{breve-2}, we can split the integral $I_{12}$ as follows:
	\begin{equation}\label{breve-4}
	\begin{split}
	I_{12}&\leq\int_{-\infty}^{\infty}
	\left|\int_{-\infty}^{\xi_2}
	\left(e^{y_1(t,\xi)-x}-e^{y_2(t,\xi)-x}\right)
	m_{0,1}(y_0(\xi))\pxi y_0(\xi)\,d\xi\right|dx\\
	&\quad+\int_{-\infty}^{\infty}
	\left|\int_{-\infty}^{\xi_2}
	e^{y_2(t,\xi)-x}(m_{0,1}-m_{0,2})(y_0(\xi))\pxi y_0(\xi)
	\,d\xi\right|dx=I_{12,1}+I_{12,2}.
	\end{split}
	\end{equation}
	Here $I_{12,1}$ and $I_{12,2}$ can be estimated as follows:
	\begin{equation*}
	\begin{split}
	I_{12,1}&\leq
	\int_{-\infty}^{\infty}
	\left|e^{y_1(t,\xi)}-e^{y_2(t,\xi)}\right|
	|m_{0,1}(y_0(\xi))\pxi y_0(\xi)|
	\int_{y_2(t,\xi)}^{\infty}e^{-x}\,dx\,d\xi\\
	&\leq \|m_{0,1}\|_{L^1}
	e^{2\max\{\|\zeta_1(t,\cdot)\|_C,
		\|\zeta_2(t,\cdot)\|_C\}}
	\|\zeta_1(t,\cdot)-\zeta_2(t,\cdot)\|_C,
	\end{split}
	\end{equation*}
	and 
	\begin{equation*}
	\begin{split}
	I_{12,2}&\leq
	\int_{-\infty}^{\infty}
	e^{y_2(t,\xi)}
	|(m_{0,1}-m_{0,2})(y_0(\xi))\pxi y_0(\xi)|
	\int_{y_2(t,\xi)}^{\infty}e^{-x}\,dx\,d\xi
	\leq\|m_{0,1}-m_{0,2}\|_{L^1}.
	\end{split}
	\end{equation*}
	Combining \eqref{breveJ-diff-2}, \eqref{breve-3}, \eqref{breve-4}, 
	\eqref{s-zeta-a} and \eqref{Lip-cont-zeta-a}, we obtain
	\begin{equation*}
	\left\|\breve{J}_{1,1}(m_{0,1})-\breve{J}_{1,2}(m_{0,2})\right\|_{L^1}
	\leq C_0\left(\|m_{0,1}-m_{0,2}\|_{L^1}
	+\|n_{0,1}-n_{0,2}\|_{L^1}\right),
	\end{equation*}
	where $C_0=C_0(T,R_0)>0$.
	Arguing similarly for 
	$\breve{J}_{2,1}(m_{0,1})-\breve{J}_{2,2}(m_{0,2})$ and
	$\breve{J}_{i,1}(n_{0,1})-\breve{J}_{i,2}(n_{0,2})$, $i=1,2$, we 
	eventually arrive at 
	\begin{equation}\label{breveJ-diff-L1}
	\begin{split}
	&\left\|\breve{J}_{i,1}(t,\cdot;m_{0,1})
	-\breve{J}_{i,2}(t,\cdot;m_{0,2})\right\|_{L^1}
	\leq C_0(\|m_{0,1}-m_{0,2}\|_{L^1}+\|n_{0,1}-n_{0,2}\|_{L^1}),
	\quad i=1,2,\\
	&\left\|\breve{J}_{i,1}(t,\cdot;n_{0,1})
	-\breve{J}_{i,2}(t,\cdot;n_{0,2})\right\|_{L^1}
	\leq C_0(\|m_{0,1}-m_{0,2}\|_{L^1}+\|n_{0,1}-n_{0,2}\|_{L^1})
	,\quad\,\,\, i=1,2,
	\end{split}
	\end{equation}
	for all $t\in[-T,T]$ and some $C_0=C_0(T,R_0)>0$.
	Combining \eqref{breveJ-diff-L1} and \eqref{Ju-j}, we 
	obtain \eqref{Lip-cont-sol-a} for 
	$C\left([-T,T],W^{1,1}\right)$
	and thus we have proved \eqref{Lip-cont-sol-a}.
	
	Then successively differentiating \eqref{Ju-j_x} with 
	respect to $x$ and using \eqref{breveJ-x}
	together with \eqref{s-zeta}, \eqref{Lip-cont-zeta} and
	\eqref{breveJ-diff}, \eqref{breveJ-diff-L1}, we eventually 
	arrive at \eqref{Lip-cont-sol-b}.
\end{proof}

Finally, combining Propositions \ref{loc-sol}, \ref{Uni} and \ref{Lip-cont} we 
obtain the main result of this section about the local 
well-posedness of Cauchy problem \eqref{C-nonl-two-comp} in the class
$C\left([-T,T], X^{k+2}\right)$ with $k\in\mathbb{N}_0$:
\begin{theorem}[Local well-posedness]
	\label{lwp}
	Suppose that $u_0(x),v_0(x)\in X^{k+2}$,
	$k\in\mathbb{N}_0$ 
	(see \eqref{X-not} for the definition of $X^k$).
	Then for a sufficiently small $T>0$ there exists a unique solution $(u,v)(t,x)$
	of the Cauchy problem \eqref{C-nonl-two-comp}, which satisfies 
	\begin{equation}
	\nonumber
	u,v\in
	C\left([-T,T], X^{k+2}\right)\cap
	C^1\left([-T,T], X^{k+1}\right).
	\end{equation}
	Moreover, $u$ and $v$ can be found by \eqref{u-repr} and \eqref{v-repr} 
	respectively, where the characteristics $y(t,\xi)=\zeta(t,\xi)+\xi$ are 
	given in Proposition \ref{char-loc}.
	
	Finally, the data-to-solution map is Lipschitz continuous.
	More precisely,
	for any constant $R>0$ and two solutions $(u_j,v_j)$, $j=1,2$, with initial data $(u_{0,j},v_{0,j})$ such that
	$$
	\|m_{0,j}\|_{X^{k}},
	\|n_{0,j}\|_{X^{k}}\leq R,\quad j=1,2,
	$$
	where 
	$(m_{0,j},n_{0,j})=(1-\px^2)(u_{0,j},v_{0,j})$,
	we have 
	\begin{align}
	&\nonumber
	\|m_1-m_2\|
	_{C([-T,T],X^{k})},
	\|n_1-n_2\|_{C([-T,T],X^{k})}
	\leq C\left(\|m_{0,1}-m_{0,2}\|_{X^{k}}
	+\|n_{0,1}-n_{0,2}\|_{X^{k}}
	\right),
	\end{align}
	with 
	$(m_{j},n_{j})=(1-\px^2)(u_{j},v_{j})$ and some
	$C=C(T,R)>0$.
	In addition, the solutions $(u_j,v_j)$
	satisfy the continuity condition \eqref{Lip-cont-sol-a}.
\end{theorem}

\section{Blow up criteria}\label{blc}
We can extend the local characteristics $\zeta$ obtained in 
Proposition \ref{char-loc} to a maximal 
interval $\bigl(-\tilde{T}_{\max},T_{\max}\bigr)$, where
$0<T_{\max},\tilde{T}_{\max}\leq\infty$.
This means that for any $\tilde{T},T>0$ which satisfy
$-\tilde{T}_{\max}<-\tilde{T}<0$ and $0<T<T_{\max}$,
there exists $l=l(T,\tilde{T})>0$ such that
$\zeta(t,\xi)\in C\bigl([-\tilde{T},T],E_l\bigr)$ is a unique solution 
of the Cauchy problem \eqref{zeta-iv}.
Of course, $\zeta$ also satisfies the regularity and decay 
conditions \eqref{zeta-reg} on the interval $[-\tilde{T},T]$
and thus this interval can be used in Proposition \ref{loc-sol} for obtaining 
a unique solution $(u,v)$ of \eqref{C-nonl-two-comp} in 
$C\bigl([-\tilde{T},T], X^{k+2}\bigr)$.
Moreover, in Theorem \ref{bl-up-cr} we will prove that the maximal 
time of existence of the local solution $(u,v)$ of the 
Cauchy problem for \eqref{nonl-two-comp} is precisely
$\bigl(-\tilde{T}_{\max},T_{\max}\bigr)$. 

In the following proposition we give a criterion for the nonexistence 
of the global characteristics $\zeta$ and establish its regularity 
and decay properties up to its maximal time of existence.

\begin{proposition}[Characteristics on the maximal interval]\label{max-char}
	Assume that $m_0,n_0$ and
	$y_0$ satisfy the same conditions as in Proposition \ref{char-loc}.
	Consider $\zeta(t,\xi)$ on the maximal interval 
	$\bigl(-\tilde{T}_{\max},T_{\max}\bigr)$, with
	$0<T_{\max},\tilde{T}_{\max}\leq\infty$.
	Then $T_{\max}$ and/or $\tilde{T}_{\max}$ are finite if and only if 
	\begin{equation}\label{crit-loc-char}
	\lim\limits_{
	\substack{t\to T_{\max},\,\,and/or\\
		t\to -\tilde{T}_{\max}}}
	\Bigl(\, 
	\inf\limits_{\xi\in\mathbb{R}}
	(\pxi y(t,\xi))\Bigr)=0.
	\end{equation}
	Moreover, the characteristics $\zeta(t,\xi)$ can be uniquely 
	continued up to the blow up time in such a way that it 
	satisfies the following regularity and decay properties (cf.\,\,\eqref{zeta-reg}): 
	\begin{equation}\label{cont-char}
	\zeta(t,\xi)\in C\left(\mathcal{I}, X^1\right),
	\quad
	\pt\zeta(t,\xi)\in L^\infty
	\left(\mathcal{I}, W^{1,\infty}(\mathbb{R})\cap
	W^{1,1}(\mathbb{R})
	\right),
	\end{equation}
	for any closed and bounded
	$\mathcal{I}\subset
	\overline{\bigl(-\tilde{T}_{\max}, T_{\max}\bigr)}$.

	Finally, 
	for all $\xi^\prime$ such that 
	$\pxi y\left(T_{\max},\xi^\prime\right)=0$ or
	$\pxi y\left(-\tilde{T}_{\max},\xi^\prime\right)=0$
	we have
	\begin{equation}\label{nonz-iv}
	m_0^2(y_0(\xi^\prime))+n_0^2(y_0(\xi^\prime))>0.
	\end{equation}
\end{proposition}
\begin{proof}
	The times $T_{\max}$ and/or $\tilde{T}_{\max}$ are finite if and only if either
	$\|\zeta(t,\cdot)\|_{C^{1}(\mathbb{R})}$ blows up as $t\to T_{\max}$ and/or 
	$t\to-\tilde{T}_{\max}$
	or
	$\inf\limits_{\xi\in\mathbb{R}}\pxi y(t,\xi)=
	1+\inf\limits_{\xi\in\mathbb{R}}\pxi \zeta(t,\xi)$
	 converges to zero as $t\to T_{\max}$ and/or 
	 $t\to-\tilde{T}_{\max}$.
	Using \eqref{zeta-int}, \eqref{pdpsi-zeta} and \eqref{UZ-L1}, we have the following \textit{a priori} estimates for $\zeta$ (cf.\,\,\eqref{s-zeta}):
	\begin{equation}\label{est-C-zeta}
	\begin{split}
	&\|\zeta(t,\cdot)\|_{C}\leq
	\|y_0(\cdot)-(\cdot)\|_{C}
	+\max\{T_{\max},\tilde{T}_{\max}\}
	\|m_0\|_{L^{1}}
	\|n_0\|_{L^{1}},\\
	&\|\partial_{(\cdot)}\zeta(t,\cdot)\|_{C}\leq
	1+\|\pxi y_0\|_{C}\\
	&\qquad\qquad\qquad
	+\max\{T_{\max},\tilde{T}_{\max}\}
	\|\pxi y_0\|_{C}
	\left(
	\|m_0\|_{C}
	\|n_0\|_{L^{1}}
	+
	\|m_0\|_{L^1}
	\|n_0\|_{C}
	\right),
	\end{split}
	\end{equation}
	for $t\in\left(-\tilde{T}_{\max},T_{\max}\right)$. The latter estimates 
	imply that $\|\zeta(t,\cdot)\|_{C^{1}}$ cannot blow up in finite time, which implies 
	the blow up criteria \eqref{crit-loc-char}.
	
	The inequalities \eqref{est-C-zeta} also imply that the 
	characteristics $\zeta(t,\xi)$ and $\pxi\zeta(t,\xi)$ can be 
	continued up to the finite $T_{\max}$ and/or $-\tilde{T}_{\max}$ by 
	taking the limit in the variable $t$ in \eqref{zeta-int} and 
	\eqref{pdpsi-zeta} respectively. 
	Then inequalities \eqref{UZ-L1} as well as boundedness 
	of $\pxi y_0,m_0,n_0$ on the line yield that 
	$\pt\zeta,\zeta\in L^\infty(\mathcal{I},W^{1,\infty})$.
	Taking into account that $\pxi y_0,m_0,n_0\in C$ and that the functions 
	under the integral in \eqref{pdpsi-zeta} are continuous and uniformly bounded 
	with respect to $\xi$ for all fixed $t\in(-\tilde{T}_{\max},T_{\max})$
	(see Item 1 in Lemma \ref{U-diff} and \eqref{UZ-L1}), we have 
	by the dominated convergence theorem that 
	$\zeta\in L^\infty(\mathcal{I},C^{1})$.
	Finally, in the case $T_{\max}<\infty$ and
	$T_{\max}\in\mathcal{I}$, we split up the integrals in \eqref{zeta-int} 
	and \eqref{pdpsi-zeta} into a sum $\int_0^{T_{\max}-\ve}\,d\eta
	+\int_{T_{\max}-\ve}^{T_{\max}}\,d\eta$ for some $\ve>0$ and conclude that 
	$\zeta\in C(\mathcal{I},C^{1})$.
	The case $\tilde{T}_{\max}<\infty$ can be treated in a similar way.
	
	Now we prove the decay properties of $\zeta$ and $\pt\zeta$.
	Arguing in the same way as in Lemma \ref{U-L^1} with $\mathcal{I}$ 
	instead of $[-T,T]$ and $\|\pt\zeta(t^*,\cdot)\|_{L^{\infty}}$ instead of 
	$\|\pt\zeta(t^*,\cdot)\|_C$ in \eqref{I_1-in}, we conclude that
	$J_j(m_0),J_j(n_0)\in C(\mathcal{I},L^1)$
	and 
	$J_j(|m_0|),J_j(|n_0|)\in C(\mathcal{I},L^1)$.
	Taking into account that we have
	\eqref{UZ-J} for $U,V$ and
	$|W|\leq\frac{1}{2}(J_1(|m_0|)+J_2(|m_0|))$,
	$|Z|\leq\frac{1}{2}(J_1(|n_0|)+J_2(|n_0|))$ for not strictly monotone increasing $y$, 
	we conclude that 
	$U,W,V,Z\in L^\infty(\mathcal{I},L^1)$.
	Therefore \eqref{zeta-int} and \eqref{pdpsi-zeta} imply that $\zeta\in C(\mathcal{I},W^{1,1})$
	and $\pt\zeta\in L^{\infty}(\mathcal{I},W^{1,1})$.
	
	Finally, let us prove \eqref{nonz-iv}.
	Suppose that 
	$m_0(y_0(\xi^\prime))=n_0(y_0(\xi^\prime))=0$.
	Then \eqref{pdpsi-zeta} implies that
	$\pxi y(t,\xi^\prime)=\pxi y_0(\xi^\prime)$ for
	all $t\in(-\tilde{T}_{\max},T_{\max})$, which contradicts the assumption that 
	$\pxi y(T_{\max},\xi^\prime)=0$ or
	$\pxi y(-\tilde{T}_{\max},\xi^\prime)=0$.
\end{proof}
\begin{remark}\label{f-z-ch}
	Notice that since $\pxi y(t,\xi)\to 1$, $\xi\to\pm\infty$, for all $t\in\mathcal{I}$, 
	the functions $\pxi y(T_{\max},\xi)$ and/or
	$\pxi y(-\tilde{T}_{\max},\xi)$ can be zero at the finite $\xi$ only.
\end{remark}

\begin{remark}
	Observe that the regularity and decay properties \eqref{cont-char} 
	of the characteristics on the time interval which can include the blow up time, 
	are weaker than that for the local characteristics, see \eqref{zeta-reg}.
	We lose the regularity because at $t=T_{\max}$ and/or $t=-\tilde{T}_{\max}$ 
	the characteristics $y$ are, in general, not strictly monotone 
	increasing and thus $W(t,\cdot), Z(t,\cdot)\not\in C(\mathbb{R})$ and 
	$W,Z\not\in C(\mathcal{I}, L^\infty(\mathbb{R}))$ 
	(cf.\,\,\cite[Theorem 1.1, Item (i) and Lemma 3.1]{GL18}).
	\begin{proof}
	Let us show that $W(t,\cdot)$ and $Z(t,\cdot)$ defined by \eqref{WZ} 
	are, in general, discontinuous for not strictly monotone increasing $y(t,\cdot)$.
	We give a proof for $W$, the function $Z$ can be analyzed similarly.
	Suppose that $y(t,\xi)$ is strictly monotone increasing for 
	$\xi\in(-\infty,a)\cup(b,\infty)$ and it is constant for $\xi\in[a,b]$.
	Denoting 
	$\tilde{m}_0(\xi)=m_0(y_0(\xi))\peta y_0(\xi)$, we have from \eqref{WZ}
	\begin{equation*}
	\begin{split}
	&W(t,\xi)=-\frac{1}{2}\int_{-\infty}^\xi
	e^{y(t,\eta)-y(t,\xi)}\tilde{m}_0(\eta)\,d\eta
	+\frac{1}{2}\int_{\xi}^\infty
	e^{y(t,\xi)-y(t,\eta)}\tilde{m}_0(\eta)\,d\eta,
	\quad\xi<a,\\
	&W(t,\xi)=-\frac{1}{2}\int_{-\infty}^a
	e^{y(t,\eta)-y(t,\xi)}\tilde{m}_0(\eta)\,d\eta
	+\frac{1}{2}\int_{b}^\infty
	e^{y(t,\xi)-y(t,\eta)}\tilde{m}_0(\eta)\,d\eta,
	\quad a<\xi<b,
	\end{split}
	\end{equation*}
	which implies that, in general, $W(t,a-)\neq W(t,a+)$.
		
	Now we prove that $W\not\in C([t_0-\ve,t_0], L^\infty)$, for some $t_0\in\mathbb{R}$, $\ve>0$, where 
	$y(t_0,\cdot)$ is strictly monotone increasing for $\xi\in(-\infty,a)\cup(b,\infty)$ and is constant for $\xi\in[a,b]$, while $y(t,\cdot)$ is strictly monotone increasing for all $t\in[t_0-\ve,t_0)$
	(the proof for $Z$ is the same).
	For all $\xi\in(a,b)$ we have
	$$
	|W(t_0,\xi)-W(t,\xi)|=
	|I_{13}(t,t_0,\xi)+I_{14}(t,t_0,\xi)
	+I_{15}(t,t_0,\xi)|,
	$$
	where (we drop the arguments of $I_j$, $j=13,14,15$ for simplicity)
	\begin{equation}\label{I-3-I-5}
	\begin{split}
	&I_{13}=-\frac{1}{2}\int_{-\infty}^a
	\left(e^{y(t_0,\eta)-y(t_0,\xi)}
	-e^{y(t,\eta)-y(t,\xi)}
	\right)
	\tilde{m}_0(\eta)\,d\eta,\\
	&I_{14}=\frac{1}{2}\int_{b}^\infty
	\left(e^{y(t_0,\xi)-y(t_0,\eta)}
	-e^{y(t,\xi)-y(t,\eta)}
	\right)
	\tilde{m}_0(\eta)\,d\eta,\\
	&I_{15}=\frac{1}{2}\int_a^\xi
	e^{y(t,\eta)-y(t,\xi)}\tilde{m}_0(\eta)\,d\eta
	-\frac{1}{2}\int_\xi^b
	e^{y(t,\xi)-y(t,\eta)}\tilde{m}_0(\eta)\,d\eta.
	\end{split}
	\end{equation}
	Equations \eqref{I-3-I-5} imply that
	$\|I_j(t,t_0,\cdot)\|_{L^\infty[a,b]}\to0$ as
	$t\to t_0$, $j=13,14$, while 
	$\|I_{15}(t,t_0,\cdot)\|_{L^\infty[a,b]}$ has, in general, nonzero limit as $t\to t_0$.
	\end{proof}
\end{remark}
\begin{remark}
	If the characteristics $y(T_{\max},\cdot)$ and/or
	$y(-\tilde{T}_{\max},\cdot)$ are strictly monotone increasing,
	then $\zeta$ satisfies the regularity and decay properties 
	\eqref{zeta-reg} up to the blow up time.
	The proof proceeds along the same lines as demonstrated 
	in Proposition \ref{char-loc}.
\end{remark}
Now we can establish the blow up criteria for the local 
solution of the Cauchy problem \eqref{C-nonl-two-comp}.
\begin{theorem}[Blow up criteria]\label{bl-up-cr}
	Suppose that $m_0(x),n_0(x)\in X^0$.
	Consider $\zeta(t,\xi)$, obtained in Proposition \ref{max-char}, 
	on the maximal interval $\bigl(-\tilde{T}_{\max},T_{\max}\bigr)$, with
	$0<T_{\max},\tilde{T}_{\max}\leq\infty$.
	
	If $T_{\max}$ and/or $\tilde{T}_{\max}$ are finite, then we have
	\begin{equation}\label{bl-up-1}
	\lim\limits_{\substack{t\to T_{\max},\,\,and/or\\
			t\to -\tilde{T}_{\max}}}
	\left(\|m(t,\cdot)\|_{C}+\|n(t,\cdot)\|_{C}
	\right)=\infty,
	\end{equation}
	where $(m,n)(t,x)=(1-\px^2)(u,v)(t,x)$ with $(u,v)(t,x)$ being the 
	unique solution of the Cauchy problem \eqref{C-nonl-two-comp}
	in $C\left([-\tilde{T},T], X^0\right)$ for any
	$-\tilde{T}_{\max}<-\tilde{T}<0$ and
	$0<T<T_{\max}$.
	
	Moreover, the following conditions are equivalent
	\begin{enumerate}[(I)]
	\item 
	$
	\limsup\limits_{\substack{t\to T_{\max},\,\,and/or\\
			t\to -\tilde{T}_{\max}}}
	\left(\|m(t,\cdot)\|_{C}+\|n(t,\cdot)\|_{C}
	\right)=\infty;
	$
	\item 
	$
	\lim\limits_{
		\substack{t\to T_{\max},\,\,and/or\\
			t\to -\tilde{T}_{\max}}}
	\Bigl(\,
	\inf\limits_{\xi\in\mathbb{R}}
	(\pxi y(t,\xi))\Bigr)=0;
	$
	\item $\limsup\limits_{
		\substack{t\to T_{\max},\,\,and/or\\
			t\to -\tilde{T}_{\max}}}
	\Bigl(\,
	\sup\limits_{x\in\mathbb{R}}\left[
	\left(
	(\px u)n-un+(\px v)m+vm
	\right)(t,x)
	\right]\Bigr)=\infty$;
	\item $\int_{0}^{T_{\max}}
	\|m(t,\cdot)\|_C+\|n(t,\cdot)\|_C\,dt=\infty$
	and/or
	$\int_{0}^{-\tilde{T}_{\max}}
	\|m(t,\cdot)\|_C+\|n(t,\cdot)\|_C\,dt=\infty$.
	\end{enumerate}
\end{theorem}
\begin{proof}
	Taking  the characteristics $y$ on the maximal interval 
	$\bigl(-\tilde{T}_{\max},T_{\max}\bigr)$ in the representation 
	\eqref{uv-repr} (see Theorem \ref{lwp}), we obtain the local solution 
	on the interval $[-\tilde{T},T]$ with any  $-\tilde{T}_{\max}<-\tilde{T}<0$
	and $0<T<T_{\max}$.
	Suppose that $T_{\max}<\infty$.
	Remark \ref{f-z-ch} implies that there exists $\xi^\prime\in\mathbb{R}$ 
	such that $\pxi y(T_{\max},\xi^\prime)=0$.
	Since $(u,v)$ admits the representation \eqref{uv-repr}, the 
	equalities \eqref{cons-uv} hold for all fixed 
	$t\in\bigl(-\tilde{T}_{\max}, T_{\max}\bigr)$ which, together 
	with \eqref{nonz-iv}, imply that either
	$|m(t,y(t,\xi^\prime))|\to\infty$ or 
	$|n(t,y(t,\xi^\prime))|\to\infty$ as $t\to T_{\max}$.
	Arguing in the same way in the case 
	$\tilde{T}_{\max}<\infty$, we arrive at \eqref{bl-up-1}.
	
	Now let us prove that the statements (I)--(IV) are equivalent.
	We will prove that (I) $\Rightarrow$ (II) $\Rightarrow$ (III) $\Rightarrow$ (I) and
	(II) $\Rightarrow$ (IV) $\Rightarrow$ (I).
	
	(I) $\Rightarrow$ (II).
	Since the right hand side of \eqref{cons-uv} is finite for all $t\in\bigl(-\tilde{T}_{\max}, T_{\max}\bigr)$, we conclude
	that (II) holds.
	
	(II) $\Rightarrow$ (III).
	Since $\pt y(t,\xi)=
	\left(\px u-u\right)(\px v+v)(t,y(t,\xi))$, we have
	\begin{equation*}
	\pt\pxi y(t,\xi)=
	-\left(
	(\px u)n-un+(\px v)m+vm
	\right)(t,y(t,\xi))\pxi y(t,\xi),
	\end{equation*}
	for $t,\xi\in
	\bigl(-\tilde{T}_{\max}, T_{\max}\bigr)
	\times\mathbb{R}$. This implies that
	\begin{equation}\label{y-exp}
	\pxi y(t,\xi)=\pxi y_0(\xi)
	\exp\left\{-\int_0^t
	\left((\px u)n-un+(\px v)m+vm
	\right)(\tau,y(\tau,\xi))\,d\tau
	\right\},
	\end{equation}
	which, together with (II), yields (III).
	
	(III) $\Rightarrow$ (I).
	This follows from 
	$$
	 \limsup\limits_{
		\substack{t\to T_{\max},\,\,and/or\\
			t\to -\tilde{T}_{\max}}}
	\left(
	\sup\limits_{x\in\mathbb{R}}\left\{
	\left(
	|(\px u)n|+|un|+|(\px v)m|+|vm|
	\right)(t,x)
	\right\}\right)=\infty,
	$$
	and \eqref{UZ-L1}.
	
	(II) $\Rightarrow$ (IV).
	Assume that $T_{\max}<\infty$.
	Then from \eqref{y-exp} we conclude that
	\begin{equation}\label{lim-int-1}
	\lim\limits_{
		\substack{t\to T_{\max}}}
	\left(
	\sup\limits_{\xi\in\mathbb{R}}\left\{
	\int_0^t\left(
	|(\px u)n|+|un|+|(\px v)m|+|vm|
	\right)(\tau,y(\tau,\xi))\,d\tau
	\right\}\right)=\infty.
	\end{equation}
	Using \eqref{UZ-L1} we obtain form \eqref{lim-int-1}
	$$
	\max\left\{\|m_0\|_{L^1}, \|n_0\|_{L^1}\right\}
	\int_0^{T_{\max}}(\|m(t,\cdot)\|_C+\|n(t,\cdot)\|_C)\,dt
	=\infty.
	$$
	Arguing similarly in the case $\tilde{T}_{\max}<\infty$, we arrive at (IV).
	
	(IV) $\Rightarrow$ (I).
	Follows form the fact that
	$\|m(t,\cdot)\|_C$ and $\|n(t,\cdot)\|_C$ are finite for all
	$t\in\bigl(-\tilde{T}_{\max}, T_{\max}\bigr)$.
\end{proof}
\begin{remark}
	The blow up criteria established in Theorem \ref{bl-up-cr} 
	generalize \cite[Theorem 3.2]{GL18}, where similar results 
	were obtained for the Cauchy problem for the FORQ equation (where $u=v$) 
	with initial data $m_0\in X^k$, $k\in\mathbb{N}$, having compact support.
	Also notice that Item (III) in Theorem \ref{bl-up-cr} was previously 
	obtained in \cite[Theorem 4.2]{YQZ15} for $m(t,\cdot),n(t,\cdot)
	\in H^{s}(\mathbb{R})$, $s>\frac{1}{2}$
	(see also \cite[Theorem 4.2]{WY23} for the two-component system 
	with high order nonlinearity and \cite[Theorem 4.3]{GLOQ13} for the FORQ equation).
	Finally, for the solution $m(t,\cdot),n(t,\cdot)\in H^{s}(\mathbb{R})$, 
	$s>\frac{1}{2}$, it was established in \cite[Theorem 4.1]{YQZ15} (see also
	\cite[Theorem 4.1]{WY23} and \cite[Theorem 4.2]{GLOQ13}),
	that if $T_{\max}<\infty$, then
	$$
	\int_0^{T_{\max}}\left(\|m(t,\cdot)\|_C^2
	+\|n(t,\cdot)\|_C^2\right)\,dt
	=\infty.
	$$
	The latter condition is weaker than that in 
	Theorem \ref{bl-up-cr}, Item (IV) obtained for $m(t,\cdot),n(t,\cdot)\in X^0$.
\end{remark}
\begin{remark}\label{max-T}
	Theorem \ref{bl-up-cr} implies that the maximal 
	time interval $\left(-\tilde{T}_{\max},T_{\max}\right)$
	of the solution $(u,v)$ with $u_0,v_0\in X^{k+2}$, $k\in\mathbb{N}_0$,
	does not depend on the regularity index $k$ 
	(cf.\,\,\cite[Remark 4.1]{YQZ15} 
	and \cite[Remark 4.1]{GLOQ13}).
	Indeed, consider the solution 
	$(u^{\prime},v^{\prime})$ in $X^{k^{\prime}+2}$, 
	$k^{\prime}\in\mathbb{N}_0$, $k^{\prime}<k$, 
	on the maximal interval
	$\left(-\tilde{T}_{\max}^{\prime},
	T_{\max}^{\prime}\right)$ with the same initial data $u_0,v_0\in X^k$.
	Since $k^{\prime}<k$ we have that
	$\left(-\tilde{T}_{\max},T_{\max}\right)
	\subseteq
	\left(-\tilde{T}_{\max}^{\prime},
	T_{\max}^{\prime}\right)$ and, due to the uniqueness, 
	$(u,v)=(u^{\prime},v^{\prime})$ on
	$\left(-\tilde{T}_{\max},T_{\max}\right)$.
	If, for example, $T_{\max}<T_{\max}^{\prime}$, then
	$u^{\prime},v^{\prime}\in 
	C([0,T_{\max}],X^{k^{\prime}+2})$ and thus
	$\|m(t,\cdot)\|_C+\|n(t,\cdot)\|_C<\infty$ as $t\to T_{\max}$.
	Theorem \ref{bl-up-cr} implies that
	$\inf\limits_{\xi\in\mathbb{R}}\pxi y(T_{\max},\xi)>0$
	and therefore the solution $u,v$ can be continued 
	beyond $T_{\max}$ in the class $X^{k+2}$.
	Arguing similarly for
	$\tilde{T}_{\max}^{\prime}$, we conclude that
	$\tilde{T}_{\max}=\tilde{T}_{\max}^{\prime}$ and $T_{\max}=T_{\max}^{\prime}$.
\end{remark}

In conclusion of this section, we elucidate the local-in-space 
sufficient condition that precipitates the finite time blow-up of 
the solution pair $(u,v)$. This condition was initially identified in 
the context of the two-component system \eqref{two-comp}, 
accommodating initial data $m_0,n_0$ in $H^s\cap L^1$ for $s>\frac{1}{2}$, 
as demonstrated in \cite[Theorem 4.3]{YQZ15}. Subsequent 
corroborations and extensions of this result can be found 
in \cite[Theorem 5.1]{FGLQ13}, \cite[Theorem 4.2]{GL18}, \cite[Theorem 5.1]{WY23}, 
and \cite[Theorem 5.2]{GLOQ13}. We extend these findings to 
solutions in the space $X^{k}$, see \eqref{X-not}.

\begin{theorem}\cite[Theorem 4.3]{YQZ15}.
	Assume that $m_0(x),n_0(x)\in X^k$, $k\in\mathbb{N}_0$,
	$m_0(x),n_0(x)\geq0$ for all $x\in\mathbb{R}$ and
	there exists $x_0\in\mathbb{R}$ such that
	$m_0(x_0),n_0(x_0)>0$.
	Consider the corresponding solution $(u,v)(t,x)$ 
	of \eqref{C-nonl-two-comp} on the maximal interval
	$\left(-\tilde{T}_{\max},T_{\max}\right)$ and let
	\begin{equation}\label{t-j}
	t_j=\frac{-M_0+(-1)^j\sqrt{M_0^2-2L_0N_0}}{L_0N_0}
	,\quad j=1,2,
	\end{equation}
	where
	\begin{equation*}
	\begin{split}
	&M_0=-\bigl((\px u_0)n_0-u_0n_0
	+(\px v_0)m_0+v_0m_0\bigr)(x_0),\quad
	N_0=(m_0+n_0)(x_0),\\
	&L_0=\frac{3}{2}
	\left(\|m_0\|_{L^1}+\|n_0\|_{L^1}\right)^3.
	\end{split}
	\end{equation*}
	
	Then we have
	\begin{itemize}
	\item if $M_0<-\sqrt{2L_0N_0}$, the maximal existence time $T_{\max}>0$ is finite and it has the following upper bound:
	\begin{equation*}
	T_{\max}\leq t_1.
	\end{equation*}
	In the case $T_{\max}=t_1$, we have the following estimates 
	for the blow up rate:
	\begin{equation}\label{blupr}
	\|m(t,\cdot)\|_{C}+\|n(t,\cdot)\|_C\geq
	\frac{2}{t_2L_0(T_{\max}-t)},\quad
	t\in\left(0,T_{\max}\right),
	\end{equation}
	and
	\begin{equation}\label{blupy}
	\inf\limits_{\xi\in\mathbb{R}}\pxi y(t,\xi)
	\leq t_2\frac{L_0}{2}(m_0+n_0)(x_0)(T_{\max}-t),
	\quad t\in\left(0,T_{\max}\right).
	\end{equation}
	\item If $M_0>\sqrt{2L_0N_0}$, the 
	maximal existence time $-\tilde{T}_{\max}<0$ is finite and 
	it has the following lower bound:
	\begin{equation*}
	-\tilde{T}_{\max}\geq t_2.
	\end{equation*}
	In the case $\tilde{T}_{\max}=t_2$, we have the following estimates 
	for the blow up rate:
	$$
	\|m(t,\cdot)\|_{C}+\|n(t,\cdot)\|_C\geq
	\frac{2}{|t_1|L_0(t+\tilde{T}_{\max})},\quad
	t\in\bigl(-\tilde{T}_{\max},0\bigr),
	$$
	and
	$$
	\inf\limits_{\xi\in\mathbb{R}}\pxi y(t,\xi)
	\leq |t_1|\frac{L_0}{2}(m_0+n_0)(x_0)
	(t+\tilde{T}_{\max}),
	\quad t\in\bigl(-\tilde{T}_{\max},0\bigr).
	$$
	\end{itemize}
\end{theorem}
\begin{proof}
The proof closely follows the methodology in \cite{YQZ15}, with 
minor modifications tailored to our specific context. Here, we provide 
a concise overview of the essential steps, highlighting where our approach 
diverges from that of \cite{YQZ15}. 
Taking into account Remark \ref{max-T}, we can assume that $k\geq 3$.
Let us take $y_0(\xi)=\xi$ and denote
$$
M(t,x)=-\bigl((\px u)n-un
+(\px v)m+vm\bigr)(t,x),\quad
N(t,x)=(m+n)(t,x).
$$
Direct calculations show that, cf.\,\,\cite[Lemma 4.5]{YQZ15} (here 
we drop the arguments of $M$, $u$ and $v$ for simplicity)
\begin{equation*}
\begin{split}
&\pt M(t,x)
-\left((uv-(\px u)\px v)-((\px u)v-u\px v)\right)\px M\\
&=-M^2-n(1-\px^2)^{-1}\left((u-\px u)M\right)
+m(1-\px^2)^{-1}\left((v+\px v)M\right)\\
&\quad+n\px(1-\px^2)^{-1}\left((u-\px u)M\right)
+m\px(1-\px^2)^{-1}\left((v+\px v)M\right).
\end{split}
\end{equation*}
Then arguing similarly as in \cite[Theorem 4.3, equation (4.38)]{YQZ15}, we conclude that
\begin{equation}\label{dM}
\begin{split}
\frac{d}{dt}M(t,y(t,x_0))&=\pt M(t,y(t,x_0))
+(W-U)(Z+V)(t,x_0)\px M(t,y(t,x_0))\\
&\leq -M^2(t,y(t,x_0))+L_0N(t,y(t,x_0)),
\quad t\in\left(-\tilde{T}_{\max},T_{\max}\right),
\end{split}
\end{equation}
and (see \cite[equation (4.39)]{YQZ15})
\begin{equation}\label{dN}
\frac{d}{dt}N(t,y(t,x_0))=-(MN)(t,y(t,x_0)),\quad
t\in\left(-\tilde{T}_{\max},T_{\max}\right).
\end{equation}
From the assumptions of the theorem and \eqref{cons-uv}, 
$N(t,y(t,x_0))>0$ for all $t\in\bigl(-\tilde{T}_{\max},T_{\max}\bigr)$.
Combining \eqref{dM} and \eqref{dN}, we conclude that
\begin{equation*}
	\left(
	N\frac{d}{dt}M-M\frac{d}{dt}N
	\right)(t,y(t,x_0))\leq L_0 N^2(t,y(t,x_0)),
\end{equation*}
and thus 
\begin{equation}\label{dM/N}
\frac{d}{dt}\left(
\frac{M}{N}\right)(t,y(t,x_0))\leq L_0.
\end{equation}
Integrating the latter from $0$ to $t$ with $t>0$, we obtain
\begin{equation}\label{M-ineq}
M(t,y(t,x_0))\leq\left(\frac{M_0}{N_0}
+L_0t\right)N(t,y(t,x_0)).
\end{equation}
Combining \eqref{dN} and \eqref{M-ineq}, we obtain
\begin{equation}\label{dN^-1}
\frac{d}{dt}\left(
N^{-1}\right)(t,y(t,x_0))\leq\frac{M_0}{N_0}
+L_0t,
\end{equation}
which, after integration from $0$ to $t$, $t>0$, leads to 
\begin{equation}\label{N^-1}
0<N^{-1}(t,y(t,x_0))\leq\frac{L_0}{2}(t-t_1)(t-t_2),
\end{equation}
where $t_1$ and $t_2$
 are the solutions of the quadratic equation 
$t^2+\frac{2M_0}{L_0N_0}t+\frac{2}{L_0N_0}=0$ given in \eqref{t-j}.
In view of the assumption $M_0<-\sqrt{2L_0N_0}$, we have that $0<t_1<t_2$ which, together with \eqref{N^-1}, implies that $T_{\max}\leq t_1$ and
$\|m(t,\cdot)\|_C+\|n(t,\cdot)\|_C\to\infty$ as 
$t\to T_{\max}$.
The blow up rate \eqref{blupr} follows from \eqref{N^-1} and the inequality 
$$
\|m(t,\cdot)\|_C+\|n(t,\cdot)\|_C\geq
N(t,y(t,x_0)),
$$
while the estimate \eqref{blupy} follows from \eqref{N^-1} and (see \eqref{cons-uv}; 
cf.\,\,\cite[Theorem 4.2]{GL18})
$$
\inf\limits_{\xi\in\mathbb{R}}\pxi y(t,\xi)\leq
\pxi y(t,x_0)=\frac{(m_0+n_0)(x_0)}{N(t,y(t,x_0))}.
$$

Arguing similarly for $-\tilde{T}_{\max}$, where we integrate \eqref{dM/N} and \eqref{dN^-1} from $t$ to $0$ with $t<0$, we obtain the lower bound for $-\tilde{T}_{\max}$ as well as the blow up rate.
\end{proof}


\begin{thebibliography}{99}
	\bibitem{AM13}
	M.J. Ablowitz and Z.H. Musslimani.
	\newblock Integrable nonlocal nonlinear 
	Schr{\"o}dinger equation.
	\newblock {\em Phys. Rev. Lett.}, 110:064105, 2013.
	
	\bibitem{AM17}
	M.J. Ablowitz and Z.H. Musslimani.
	\newblock Integrable nonlocal nonlinear equations.
	\newblock {\em Stud. Appl. Math.}, 139:7--59, 2017.
	
	\bibitem{AR19}
	S.C. Anco and E. Recio.
	\newblock
	A general family of multi-peakon equations and their properties. 
	{\em J. Phys. A: Math. Theor.}, 52:125203, 2019.
	
	
	
	\bibitem{BKS20}
	A. Boutet de Monvel, I. Karpenko and D. Shepelsky.
	\newblock A {R}iemann-{H}ilbert approach to the modified
	Camassa-Holm equation with nonzero boundary
	conditions.
	\newblock{\em J. Math. Phys.}, 61:031504, 2020.
	
	\bibitem{BC07}
	A. Bressan and A. Constantin.
	\newblock Global conservative solutions of the Camassa-Holm equation.
	\newblock {\em Arch. Ration. Mech. Anal.}, 183:215--239, 2007.
	
	\bibitem{CHS16}
	X.-K. Chang, X.-B. Hu, J. Szmigielski.
	\newblock Multipeakons of a two-component modified
	Camassa-Holm equation and the relation with the
	finite Kac-van Moerbeke lattice.
	\newblock {\em Adv. Math.}, 299:1--35, 2016.
	
	\bibitem{CGLQ16}
	R.M. Chen, F. Guo, Y. Liu, C. Qu. 
	\newblock
	Analysis on the blow-up of solutions to a class of integrable peakon equations. 
	\newblock
	{\em J. Funct. Anal.}, 270(6):2343--2374, 2016.
	
	
	\bibitem{F95}
	A.S. Fokas.
	\newblock On a class of physically important integrable equations.
	\newblock {\em Physica D}, 87:145--150, 1995.
	
	\bibitem{F81}
	B. Fuchssteiner
	\newblock The Lie Algebra Structure of Nonlinear Evolution Equations Admitting Infinite Dimensional Abelian Symmetry Groups.
	\newblock {\em Progress of Theoretical Physics}, 65:861--876, 1981.
	
	\bibitem{FGLQ13}
	Y. Fu, G. Gui, Y. Liu, C. Qu.
	\newblock On the Cauchy problem for the integrable modified
	Camassa-Holm equation with cubic nonlinearity.
	\newblock
	{\em J. Differential Equations}, 255:1905--1938, 2013.
		
	\bibitem{GL18}
	Y. Gao and J-G. Liu.
	\newblock The modified Camassa-Holm equation in Lagrangian coordinates.
	\newblock {\em Discr. Cont. Dyn. Syst. Ser. B.}, 23(6):2545--2592, 2018.
	
	\bibitem{GL17}
	Y. Gao and J.-G. Liu.
	\newblock Global convergence of a sticky particle method for the modified Camassa-Holm equation. 
	\newblock {\em SIAM J. Math. Anal.}, 49: 1267--1294, 2017.
	
	\bibitem{GHR12}
	K. Grunert, H. Holden and X. Raynaud.
	\newblock
	Global solutions for the two-component Camassa–Holm system.
	\newblock {\em Comm. Partial Differential Equations}, 37(12):2245--2271, 2012.
	
	
	\bibitem{GLOQ13}
	G. Gui, Y. Liu, P.J. Olver and C. Qu.
	\newblock Wave-breaking and peakons for a modified
	Camassa-Holm equation.
	\newblock {\em Commun. Math. Phys.}, 319:731--759, 2013.
	
	\bibitem{HM14}
	A. Himonas, D. Mantzavinos. 
	The Cauchy problem for the Fokas-Olver-Rosenau-Qiao equation.
	\newblock {\em J. Nonlinear Analysis: Theory, Methods 
	\& Applications}, 95:499--529, 2014.
	
	\bibitem{HM14-1}
	A. Himonas, D. Mantzavinos.
	\newblock 
	H\"older continuity for the Fokas-Olver-Rosenau-Qiao equation.
	\newblock {\em J. Nonlinear. Sci.}, 24:1105--1124, 2014.
	
	\bibitem{HR07}
	H. Holden and X. Raynaud.
	\newblock Global conservative solutions of the 
	Camassa-Holm equation -- a Lagrangian point of view.
	\newblock {\em Comm. Partial Differential Equations}, 32:1511--1549, 2007.
	
	\bibitem{HFQ17}
	Y. Hou, E. Fan, Z. Qiao.
	\newblock
	The algebro-geometric solutions for the Fokas-Olver-Rosenau-Qiao (FORQ) hierarchy.
	\newblock
	{\em J. Geom. Phys.} 117:105--133, 2017.
	
	\bibitem{KR24}
	K.H. Karlsen, Ya. Rybalko.
	\newblock On the well-posedness of a nonlocal (two-place) FORQ equation  
	via a two-component peakon system.
	\newblock {\em J. Math. Anal. Appl.}, 529:127601, 2024.
	
	\bibitem{K22}
	I. Karpenko.
	\newblock
	Long-time asymptotics for the modified Camassa-Holm equation with nonzero boundary conditions.
	\newblock
	{\em J. Math. Phys. Anal. Geom.}, 16(4):418--453, 2022.
	
	\bibitem{LH17}
	S.Y. Lou, F. Huang.
	\newblock Alice-Bob physics{:} coherent solutions of nonlocal KdV systems.
	\newblock {\em Sci. Rep.}, 7:869, 2017.
	
	\bibitem{LQ17}
	S.Y. Lou, Z. Qiao.
	\newblock Alice-Bob peakon systems.
	\newblock {\em Chin. Phys. Lett.}, 34(10):100201, 2017.
	
	\bibitem{M14}
	Y. Matsuno.
	\newblock
	Smooth and singular multisoliton solutions
	of a modified Camassa-Holm equation with
	cubic nonlinearity and linear dispersion.
	\newblock {\em J. Phys. A: Math. Theor.}, 47:125203, 2014.
	
	\bibitem{MM13}
	Y. Mi and C. Mu.
	\newblock Well-posedness and analyticity for an integrable 
	two-component system with cubic nonlinearity.
	\newblock {\em J. Hyperbolic Differ. Equations}, 10(04):703--723, 2013.
	
	\bibitem{N09}
	V. Novikov.
	\newblock Generalizations of the Camassa-Holm equation.
	\newblock {J. Phys. A: Math. Theor.}, 42:342002, 2009.
	
	\bibitem{OR96}
	P.J. Olver and P. Rosenau.
	\newblock Tri-Hamiltonian duality between solitons and 
	solitary-wave solutions having compact support.
	\newblock {\em Phys. Rev. E}, 53:1900--1906, 1996.
	
	\bibitem{Q06}
	Z.J. Qiao.
	\newblock A new integrable equation with cuspons and W/M-shape-peaks solitons.
	\newblock {\em J. Math. Phys.}, 47:112701, 2006.
	
	\bibitem{RS21}
	Ya. Rybalko, D. Shepelsky. 
	\newblock
	Long-time asymptotics for the integrable nonlocal focusing 
	nonlinear Schr\"odinger equation for a family of step-like initial data.
	\newblock
	{\em Commun. Math. Phys.} 382(1):87--121, 2021.
	
	\bibitem{Sch96}
	J. Schiff.
	\newblock Zero curvature formulations of dual hierarchies.
	\newblock {\em J. Math. Phys.}, 37:1928, 1996.
	
	\bibitem{SW04}
	T. Sch\"afer, C.E.Wayne.
	\newblock Propagation of ultra-short optical pulses in cubic nonlinear media. 
	\newblock {\em Physica D}, 196:90--105, 2004.
	
	\bibitem{SQQ11}
	J.F. Song, C.Z. Qu and Z.J. Qiao.
	\newblock A new integrable two-component system with cubic nonlinearity.
	\newblock {\em J. Math. Phys.}, 52:013503, 2011.
	
	\bibitem{TL13}
	K. Tian and Q.P. Liu.
	\newblock Tri-Hamiltonian duality between the 
	Wadati-Konno-Ichikawa hierarchy and the Song-Qu-Qiao hierarchy.
	\newblock {\em J. Math. Phys.}, 54:043513 2013.
	
	\bibitem{WY23}
	Z. Wang and K. Yan.
	\newblock Blow-up data for a two-component 
	Camassa-Holm system with high order nonlinearity.
	\newblock {\em J. Differential Equations}, 358(15):256--294, 2023.
	
	\bibitem{XQZ15}
	B. Xia, Z. Qiao, R. Zhou.
	\newblock A synthetical two-component model with peakon solutions.
	\newblock \textit{Stud. Appl. Math.}, 135(3):248--276, 2015.
	
	\bibitem{WZ22}
	Y. Wang, M. Zhu. 
	\newblock
	On the Cauchy problem for a two-component peakon system 
	with cubic nonlinearity.
	\newblock
	{\em J. Dyn. Diff. Equat.}, 2022.
	
	\bibitem{Y20}
	K. Yan. 
	\newblock
	On the blow up solutions to a two-component cubic Camassa-Holm 
	system with peakons. 
	\newblock
	{\em Discr. Cont. Dyn. Syst.}, 40(7):4565--4576, 2020.
	
	\bibitem{YQZ15}
	K. Yan, Z. Qiao, and Y. Zhang. 
	\newblock Blow-up phenomena for an integrable two-component 
	Camassa-Holm system with cubic nonlinearity and peakon solutions.
	\newblock {\em J. Differential Equations}, 259(11):6644--6671, 2015.
	
	\bibitem{Y20}
	S. Yang
	\newblock Blow-up phenomena for the generalized FORQ/MCH equation.
	\newblock {\em Z. Angew. Math. Phys.},
	71:20, 2020.
	
	\bibitem{YC24}
	S. Yang, J. Chen.
	\newblock On the finite time blow-up for the high-order 
	Camassa-Holm-Fokas-Olver-Rosenau-Qiao equations.
	\newblock {\em J. Differential Equations},
	379:829--861, 2024.
	
	\bibitem{ZGX22}
	F. Zeng, Y. Gao, X. Xue.
	\newblock Global weak solutions to the generalized mCH equation via characteristics.
	\newblock {\em Discr. Cont. Dyn. Syst. Ser. B}, 27(8):4317-4329, 2022.
	
	\bibitem{Zh16}
	Q. Zhang.
	\newblock Global wellposedness of cubic Camassa-Holm equations. 
	\newblock {\em Nonlinear Anal.}, 133:61--73, 2016.
	
\end{thebibliography}
\end{document}